\DeclareRobustCommand{\SkipTocEntry}[4]{}
\makeindex

\documentclass[letterpaper,12pt,reqno,latexsym,amsbsy,xypic,mathrsfs,verbatim]{amsart}
\usepackage{amsfonts,amssymb,amsthm}
\usepackage{amsmath,amscd}
\usepackage{pstricks}
\usepackage{pstricks,pst-node}
\usepackage{mathrsfs}
\usepackage[all]{xy}
\usepackage[enableskew,vcentermath]{youngtab}

\renewcommand{\subsection}[1]{\vspace{.18in}\par\noindent\addtocounter{subsection}{1}\setcounter{equation}{0}{\bf\thesubsection.\hspace{5pt}#1}}

\newtheorem{theorem}{Theorem}[section]
\numberwithin{theorem}{subsection}
\numberwithin{equation}{subsection}

\theoremstyle{definition}
\newtheorem{defn}[theorem]{Definition}

\newtheorem{Example}[theorem]{Example}
\newtheorem{rem}[theorem]{Remark}

\theoremstyle{plain}
\newtheorem{prop}[theorem]{Proposition}
\newtheorem{thm}[theorem]{Theorem}
\newtheorem{lem}[theorem]{Lemma}
\newtheorem{cor}[theorem]{Corollary}

\newcommand{\ds}{\displaystyle}

\newcommand{\C}{\mathbb C}
\newcommand{\Z}{\mathbb Z}
\newcommand{\N}{\mathbb N}

\newcommand{\mf}{\mathfrak}

\newcommand{\bfa}{{\bf a}}

\newcommand{\bfs}{{\bf s}}

\newcommand{\la}{\lambda}
\newcommand{\ga}{\gamma}

\newcommand{\La}{\Lambda}

\newcommand{\Og}{\Omega}
\newcommand{\und}{\underline}
\newcommand{\udla}{\underline{\lambda}}
\newcommand{\udu}{\underline{u}}

\newcommand{\mc}{\mathcal}

\newcommand{\ot}{\otimes}

\newcommand{\vtg}{{\!\vartriangle\!}}
\newcommand{\cycn}{{{n}}}
\newcommand{\ms}{\mathscr}
\newcommand{\bfQ}{\mathbf{Q}}
\newcommand{\sH}{\mathcal{H}}
\newcommand{\fS}{\mathfrak{S}}
\newcommand{\udga}{\underline{\gamma}}

\def\sfs{{\mathsf s}}
\def\sfc{{\mathsf c}}
\def\sfr{{\mathsf r}}
\def\mfs{{\mathfrak s}}
\def\mft{{\mathfrak t}}
\def\bfc{{\mathbf c}}
\newcommand{\afgl}{\widehat{\frak{gl}}_n}
\newcommand{\Ha}{\mathcal{H}(r)}
\newcommand{\AKa}{\mc H_{\und{u}}(r)}
\newcommand{\afHr}{{\mathcal H_\vtg(r)}}
\newcommand{\afSr}{{\mathcal S}_{\vtg}(\cycn,r)}

\def\sQ{{\mathcal Q}}
\def\ttx{{\tt x}}
\def\ttg{{\tt g}}

\def\ttk{{\tt k}}

\def\rmU{{\rm U}}

\def\fks{{\mathfrak s}}

{\vskip-\lastskip\medskip
  \noindent
  {\em #1.}\enspace
  }%
{\qed\par\medskip
  }

\begin{document}

\title[Standard multipartitions and Drinfeld polynomials]{Standard multipartitions and a combinatorial affine Schur-Weyl duality}
\author[Du and Wan]{Jie Du and Jinkui Wan}
\thanks{Supported by ARC DP-120101436 and NSFC-11101031. The research was carried out while Wan was visiting UNSW Australia. The hospitality and support of UNSW are gratefully acknowledged.}

\address{(Du) School of Mathematics and Statistics,
University of New South Wales,
UNSW Sydney 2052,
Australia.
 }\email{j.du@unsw.edu.au}
\address{(Wan) School of Mathematics and Statistics, Beijing Institute of Technology,
Beijing, 100081, P.R. China. } \email{wjk302@hotmail.com}


\begin{abstract}We introduce the notion of standard multipartitions and establish a one-to-one correspondence between standard multipartitions and
irreducible representations with integral weights for the affine Hecke algebra of type $A$ with a parameter $q\in\C^*$ which is not a root of unity. We then extend the correspondence to all Kleshchev multipartitions for
Ariki-Koike algebras of integral type. By the affine Schur--Weyl duality, we further extend this to a correspondence between standard multipartitions and Drinfeld multipolynomials of integral type whose associated irreducible polynomial representations completely determine all irreducible polynomial representations for the quantum loop algebra $\rmU_q(\afgl)$.  We will see, in particular, the notion of standard multipartitions gives rise to a combinatorial description of the affine Schur--Weyl duality in terms of a column-reading vs. row reading of residues of a multipartition.

\end{abstract}

\subjclass[2010]{Primary: 20C08, 20C32, 17B37. Secondary:  20B30, 20G43.}
\keywords{affine Hecke algebra, Ariki-Koike algebra, quantum loop algebra, affine Schur-
Weyl duality, standard multipartition, Drinfeld polynomial}

\maketitle

\section{Introduction}

Representations of affine Hecke algebras $\afHr=\afHr_{\mathbb C}$ of type $A$ with a parameter $q\in\mathbb C$ link, on the one hand,
representations
of quantum loop algebras $\rmU_q(\afgl)$ of $\mathfrak{gl}_n$ and link, on the other hand, representations of Ariki-Koike algebra \cite{AK} (or cyclotomic Hecke algebra \cite{BM}) $\sH_{\udu}(r)$ with
extra parameters $\udu=(u_1,\ldots,u_m)$. Irreducible $\afHr$-modules $V_\fks$ have been classified by Zelevinsky \cite{Ze} and Rogawski \cite{Ro1} in terms of multisegments $\fks$. By the (partial) affine Schur--Weyl duality developed in \cite{DDF}, this classification determines in turn the isomorphism classes of irreducible modules of the affine $q$-Schur algebras $\afSr$,
which are inflated to irreducible $\rmU_q(\afgl)$-modules. Thus, by \cite{FM}, we may index these inflated irreducible modules by Drinfeld multipolynomials.

Irreducible $\afHr$-modules are also inflations of irreducible $\AKa$-modules which are labelled by Kleshchev multipartitions \cite{Ar}. Thus, a Kleshchev multipartition determines an irreducible $\afHr$-modules and, hence, an irreducible $\rmU_q(\afgl)$-module and then a Drinfeld multipolynomial. So it is natural to ask how a given Kleshchev multipartition of $r$ gives rise to an $n$-tuple of Drinfeld polynomials. Since this relation is a many-to-one relation, a second natural question is how to single out certain Kleshchev multipartitions to obtain a one-to-one relation. The purpose of this paper is to answer these two questions.

Our approach was motivated by the work of Vazirani \cite{Va} and Fu and the first author \cite{DF} in which they computed the Drinfeld polynomials associated with small representations defined by partitions. First, it suffices to consider the Ariki--Koike algebras whose extra parameters $u_i$ are powers of $q^2$. These algebras are called Ariki--Koike algebras of {\it integral type}.  When $q$ is not a root of unity, Kleshchev multipartitions relative to such an Ariki--Koike algebra have an easy characterization following \cite{Va,BK}.
Second, it is enough to consider the irreducible $\afHr$-modules indexed by integral multisegments.
Following \cite{Ro1}, we order an integral multisegment in a unique way so that it defines a certain so-called standard words. This motivated us to introduce {\it standard} multipartitions. More precisely, given an integral segment $\mfs$, we construct a unique standard multipartition $\udga_{\mfs}$ which is actually a Keshchev multipartition and hence there exists an irreducible representation $D^{\udga'_{\mfs}}$ of an Ariki-Koike algebra of integral type.
Then  we show $V_{\mfs}\cong D^{\udga'_{\mfs}}$ and moreover $\mfs$ can be obtained by reading the column residual segments of $\udga'_{\mfs}$.
In this way, we solve the one-to-one relation problem between standard multipartitions and integral multisegments.

We further extend the construction for the one-to-one relation to a many-to-one relation
by taking advantage of the connection between Kleshchev multipartitions and multisegments given in \cite{Va}.
Our key observation is that, for an arbitrary integral segment $\fks$, twisting the irreducible $\afHr$-module $M_\mfs$ considered in \cite{Va} by Zelevinsky's involution results in, up to isomorphism, the irreducible module $V_{\mfs}$ constructed in \cite{Ro1}; see Lemma~\ref{lem:VsMs}.

Next, we extend the construction, by the affine Schur--Weyl duality developed in \cite{DDF} (see \cite{CP} for the case concerning the quantum affine algebra $\rmU_q(\widehat{\frak{sl}}_n)$), to
irreducible representations of the quantum loop algebra $\rmU_q(\afgl)$.  By introducing Drinfeld polynomials of integral type, we prove that tensoring $D^{\udla}$ with the tensor space gives an irreducible representation of $\rmU_q(\afgl)$ whose Drinfeld polynomials have roots associated with the row residual segments of $\udla$ for a multipartition $\udla$ dual to a Kleshchev multipartition. Like the Hecke algebra case, we show in Theorem \ref{thm:intDrinfeld} that irreducible polynomial $\rmU_q(\afgl)$-modules are completely determined by those associated with Drinfeld polynomials of integral type, which recovers an observation in \cite{HL}.
This is another interesting application of affine $q$-Schur algebras and the Schur--Weyl duality.
Moreover, the symmetry between column-reading of residual multisegments and row-reading for roots of Drinfeld polynomials reveals a key role played by standard multipartitions
in the affine  Schur--Weyl duality; see Remark \ref{cr-reading}.

We organise the paper as follows. We first briefly review in \S2 affine Hecke algebras and the classification of their irreducible modules, following Rogawski. We then introduce integral multisegments and their associated standard words. In \S3, we introduce Specht modules of Ariki--Koike algebras and explicitly work out the action of Jucys--Murphy operators on Specht modules. We further interpret the inflation of a Specht module as  an induced $\afHr$-module.
Then, we introduce Ariki--Koike algebras of integral type and their associated Kleshchev multipartitions.   With these preparations, we are ready to get the main results of the paper in the next three sections.
Standard multipartitions are introduced in \S4 and are proved to be in one-to-one correspondence with integral multisegments.
In \S5, we construct, for every irreducible module of an Ariki--Koike algebra of integral type, the corresponding multisegment
when regarded as an irreducible $\afHr$-module. Drinfeld polynomials of integral type are introduced in
\S6 where we will extend the one-to-one relation to irreducible representations of affine $q$-Schur algebras.

{\it Throughout the paper, all algebras and modules considered are defined over $\C$.
Let $\C^*=\C\setminus\{0\}$ and assume that
 $q\in\mathbb C^*$ is not a root of unity.}

\section{Irreducible representations of affine Hecke algebras}
In this section, we shall give a review on
 the classification of irreducible representations
of affine Hecke algebras,  following
Rogawski~\cite{Ro1} and Zelevinsky~\cite{Ze}. We will focus on the subcategory $\afHr$-{\sf mod}$^{\Z}$ of $\afHr$-modules with integral weights.

\subsection{The affine Hecke algebra $\afHr$.}
For $r\geq 1$, the affine Hecke algebra $\afHr$ is the associative
algebra over $\C$ generated by $T_1,\ldots, T_{r-1}$, $X_1,\ldots, X_r$ and $X_1^{-1},\ldots, X_r^{-1}$ subject to the
following relations:
\begin{align*}
(T_i-q^2)(T_i+1)=0,\quad 1\leq i\leq r-1,\\
T_iT_{i+1}T_i=T_{i+1}T_iT_{i+1},\quad T_iT_j=T_jT_i, \quad |i-j|> 1,\\
X_iX_i^{-1}=1=X_i^{-1}X_i,\quad X_iX_j=X_jX_i,\quad 1\leq i,j\leq r,\\
T_iX_iT_i=q^2X_{i+1}, X_jT_i=T_iX_j,\quad j\neq i,i+1.
\end{align*}

Let $\mf S_r$ be the symmetric group on $\{1,2,\ldots,r\}$
which is generated by the simple transpositions $s_i=(i,i+1)$
with $1\leq i\leq r-1$.
Let $\Ha$ be the subalgebra of $\afHr$ generated by $T_1,\ldots, T_{r-1}$.
Then $\Ha=\sH(\fS_r)$ is the (Iwahori-)Hecke algebra associated to $\mf{S}_r$.
For $w\in\mf S_r$ with a reduced expression $w=s_{i_1}s_{i_2}\cdots s_{i_k}$,
let $T_w=T_{i_1}T_{i_2}\cdots T_{i_k}$.
It is known that the set $\{T_w|w\in\mf S_r\}$ forms a basis of the Hecke algebra $\mc H(r)$.
Given a composition $\mu=(\mu_1,\mu_2,\ldots,\mu_{\ell})$ of $r$,
denote by $\mf S_{\mu}$
the associated Young subgroup. Let $\mc H(\mu)$ be the subalgebra of $\mc H(r)$
corresponding to $\mf S_{\mu}$ and let $\mc H_{\vtg}(\mu)=\langle\sH(\mu),X^{\pm1}_1,\ldots, X^{\pm1}_r\rangle$
be the corresponding subalgebra of  $\mc H_{\vtg}(r)$.
Note that \vspace{-1ex}
\begin{equation}\label{two iso}
\sH(\mu)\cong
\sH(\mu_1)\otimes\cdots\otimes\sH(\mu_\ell)\text{ and }\sH_{\vtg}(\mu)\cong
\sH_\vtg(\mu_1)\otimes\cdots\otimes\sH_\vtg(\mu_\ell).
\end{equation}




\subsection{\bf Irreducible $\afHr$-modules.}
The cyclic subgroup
$\langle q^2\rangle$ of $\mathbb C^*$ is isomorphic to $\mathbb Z$. We will call the elements
$L_a:=a\langle q^2\rangle$ of the quotient group $\mathbb C^*/\langle q^2\rangle$ ``lines''.
A finite consecutive subset of a line is called a {\it segment}. Thus, a
segment $\sfs$ that lies on line $L_a$ is an ordered
sequence of the form $\sfs=(aq^{2i},aq^{2(i+1)},\ldots,aq^{2j})$ for some
$i,j\in\Z$ and $i\leq j$.
Given a segment $\sfs=(aq^{2i},\ldots,aq^{2j})$,
set
$\widetilde{\sfs}=(aq^{2j},aq^{2j-2},\ldots,aq^{2i})$ and $|\sfs|=j-i+1$ which is called the length of $\sfs$.
Denote by $\mathscr C$ the set of segments.
Given a sequence of segments $\bfs=(\sfs_1, \sfs_2,\ldots, \sfs_t)\in\mathscr C^t$ for $t\geq 1$ with $r=\sum_{i=1}^t|\sfs_i|$,
let
\begin{equation}\label{tilde s}
\bfs^\vee=\sfs_1\vee\cdots\vee\sfs_t\;\quad(\text{resp.},
\widetilde\bfs^\vee=\widetilde{\sfs}_1\vee\cdots\vee\widetilde{\sfs}_t)
\end{equation}
 be the $r$-tuple obtained by juxtaposing the sequences\footnote{Here we changed the notation $\chi(\bfs)$ in \cite{Ro1} to $\bfs^\vee$ to be consistent with the notation $\la^\vee$ in \S3.1.}
$\sfs_1,\ldots,\sfs_t$ (resp., $\widetilde{\sfs}_1,\ldots,\widetilde{\sfs}_t$)
and set
$$\mu(\bfs)=(|\sfs_1|, |\sfs_2|,\ldots, |\sfs_t|).$$
Given
$\bfs=(\sfs_1,\ldots,\sfs_t),\bfs'=(\sfs'_1,\ldots,\sfs'_t)\in\mathscr C^t$ for some $t\geq 1$,
we define an equivalence relation $\bfs\thicksim\bfs'$
if $\bfs$ and $\bfs'$ are equal up to a rearrangement, that is, there exists
$\sigma\in\mf S_t$ such that $\bfs'_k=\bfs_{\sigma(k)}$ for $1\leq k\leq t$.
The equivalent class of $\bfs=(\sfs_1,\ldots,\sfs_t)$, denoted by $\bar{\bfs}=\{\sfs_1,\ldots,\sfs_t\}$, is  called a {\em multisegment}.
For $t\geq 1$, set
$$
\mathscr C^t_r=\{\bfs=(\sfs_1,\ldots,\sfs_t)\in\mathscr C^t\mid\sum_i|\sfs_i|=r\}
$$
and let
$$
\mathscr S_r=\bigcup_{t\geq 1}( \mathscr C^t_r/\thicksim)
$$
be the set of multisegments of total length $r$.

It is well known that, for any $u\in\C^*$, there is an evaluation homomorphism
\begin{align}
{\rm ev}_u: \afHr&\twoheadrightarrow\Ha,\quad
X_1\mapsto u,\quad T_i\mapsto T_i \label{ev}
\end{align}
for $1\leq i\leq r-1$.
Given an $\Ha$-module $N$, let ${\rm ev}^*_u(N)$ denote the inflation of $N$
to an $\afHr$-module via this homomorphism ${\rm ev}_u$. In particular,  the ``trivial'' and ``sign'' modules $\bf 1$ and $-\bf1$ of $\Ha$ give two 1-dimensional $\afHr$-modules ${\rm ev}^*_u({\bf1})$ and
${\rm ev}^*_u({\bf -1})$.

For a segment $\sfs=(aq^{2i},\ldots, aq^{2(i+r-1)})$, introduce the notation
 \begin{equation}\label{trivial-sign}
\C_{\sfs}=\C_{\sfs}^-:={\rm ev}^*_{aq^{2(i+r-1)}} ({\bf -1})\quad\text{ and }\quad  \C_{\sfs}^+:={\rm ev}^*_{aq^{2i}} ({\bf1}).\end{equation}
Given $\bfs=(\sfs_1,\sfs_2,\ldots,\sfs_t)\in\mathscr C_r^t$ for some $t\geq 1$, let $\C_{\bfs}=\C_{\bfs}^-$ resp., $\C_{\bfs}^+$ be the one-dimensional $\mc H_{\vtg}(\mu(\bfs))$-module
$$
\C_{\bfs}:=\C_{\sfs_1}\otimes\cdots\otimes\C_{\sfs_t}\;\;\text{ resp., }\;\;\C_{\bfs}^+=\C_{\sfs_1}^+\otimes\cdots\otimes\C_{\sfs_t}^+
$$
and define
\begin{equation}\label{Is}
I_\bfs=I_\bfs^-:={\rm ind}^{\afHr}_{\mc H_{\vtg}(\mu(\bfs))}\C_{\bfs}\;\;\text{ and }\;\;\;I_\bfs^+:={\rm ind}^{\afHr}_{\mc H_{\vtg}(\mu(\bfs))}\C_{\bfs}^+.
\end{equation}
For a composition $\mu$ of $r$, let $w_\mu^0$ be the longest element in the Young subgroup $\mf S_\mu$
and let $E_{\mu}$ be the left cell module of $\Ha$ containing $w_\mu^0$.
Then we have the following results for $I_\bfs$ due to Rogawski \cite{Ro1} (cf. Zelevinsky \cite{Ze}). See \S5 for discussion on $I_\bfs^+$.

\begin{prop}[{\cite[Theorem 5.2]{Ro1}}]
\label{prop:Ro2}
Suppose $\bfs,\bfs' \in \mathscr C^t_r$ for $t\geq 1$.
Then
\begin{enumerate}


\item As an $\afHr$-module, $I_\bfs$ has a unique composition factor $V_{\bfs}$
such that as an $\Ha$-module, $E_{\mu(\bfs)}$ appears as a constituent
of $V_\bfs$.

\item $V_\bfs\cong V_{\bfs'}$ if and only if $\bfs\thicksim\bfs'$.

\item Every (finite dimensional) irreducible $\afHr$-module is isomorphic to
$V_\bfs$ for some $\bfs\in\mathscr C^t_r$ with $t\geq 1$. Hence the set $\mathscr S_r$
parametrizes all irreducible $\afHr$-modules.
\end{enumerate}
\end{prop}
%

By Proposition~\ref{prop:Ro2}(3), isomorphism types of irreducible $\afHr$-modules are indexed by the set $\mathscr S_r$.
Thus, we will use the notation $V_\fks$ to denote a member in the isomorphism class labelled by
$\fks\in\mathscr S_r$.
We also note that, for $\bfs\sim\bfs'$, $I_\bfs$ is not necessarily isomorphic to $I_{\bfs'}$.
It is known from \cite{Ze} that if $\bar\bfs=\overline{\bfs'}=\fks$ then $I_\bfs$ and $I_{\bfs'}$ have the same composition factors including the unique $V_\fks$. We now look at a sufficient condition on $\bfs$ for $V_\fks$ appearing in the head of $I_\bfs$.

\begin{defn}\label{defn:order}
\begin{enumerate}
\item
Let segments
$\sfs_1=(aq^{2i_1},\ldots,aq^{2j_1})$ and $\sfs_2=(aq^{2i_2},\ldots,aq^{2j_2})$ be on the same line $L_a$.
We say that $\sfs_1$ {\em precedes} $\sfs_2$ and write $\sfs_1\preccurlyeq\sfs_2$
if either $j_1<j_2$ or  $j_1=j_2$ and $i_1\geq i_2$ (cf. \cite[Theorem 5.2]{Ro1}).

\item A {\it chain} on line $L_a$ is a sequence of segments $\sfs_1,\sfs_2,\ldots,\sfs_n$ on $L_a$ satisfying
$\sfs_{1}\preccurlyeq\sfs_{2}\preccurlyeq\ldots\preccurlyeq\sfs_{n}$. A sequence of segments $\bfs=(\sfs_1,\sfs_2,\ldots,\sfs_t)\in\mathscr C^t$ with $t\geq 1$ is said to be {\it standard} (relative to $\preccurlyeq$) if it can be obtained by juxtaposing chains $\bfc^{(1)},\ldots,\bfc^{(p)}$ on distinct lines $L_{a_1},\ldots,L_{a_p}$. We also write
$\bfs=\bfc^{(1)}\vee\cdots\vee\bfc^{(p)}.$

\item For a segment $\sfs=(aq^{2i},\ldots,aq^{2j})$ lying on line $L_a$, define
 $\sfs^{-1}=(a^{-1}q^{-2j},\ldots,a^{-1}q^{-2i})$
 a segment lying on line $L_{a^{-1}}$.
A sequence $\bfs=(\sfs_1,\sfs_2,\ldots,\sfs_t)$ of segments is said to be {\em anti-standard}
if $\bfs^{-1}:=(\sfs_1^{-1},\sfs_2^{-1},\ldots,\sfs_t^{-1})$ is standard.
\end{enumerate}
\end{defn}
\begin{rem}
Following~\cite[Definition 3.1]{Ro2},
we can define the order $\sfs_1\preccurlyeq'\sfs_2$
if either $j_1<j_2$ or  $j_1=j_2$ and $i_1\leq i_2$ for two segments
$
\sfs_1=(aq^{2i_1},\ldots,aq^{2j_1})$ and $
\sfs_2=(aq^{2i_2},\ldots,aq^{2j_2})
$
on the same line $L_a$.
Analogous to Definition~\ref{defn:order}(2), one can define the standard sequence of segments with respect to the order $\preccurlyeq'$.
\end{rem}
\begin{prop}\label{prop:Ro3}
For $\fks\in\mathscr S_r$, if $\bfs\in\fks$ is standard (with respect to $\preccurlyeq$),
then $I_\bfs$ has irreducible head ${\rm hd}(I_\bfs)$ and
$$
V_{\fks}\cong{\rm hd}(I_{\bfs}).
$$
\end{prop}
\begin{proof}
Suppose $\bfs\in\mathscr C^t$ is standard with the chain decomposition $\bfs=\bfc^{(1)}\vee\cdots\vee\bfc^{(p)}$
in the sense of Definition~\ref{defn:order}(2).
Reorder the segments in $\bfc^{(i)}$ to obtain a chain $(\bfc^{(i)})'$ with respect to the order $\preccurlyeq'$ for $1\leq i\leq p$.
Then $\bfs':=(\bfc^{(1)})'\vee\cdots\vee(\bfc^{(p)})'$ is standard with respect to the order $\preccurlyeq'$.
Then by \cite[Proposition 5.2]{Ro2} one can show that $I_{\bfs}\cong I_{\bfs'}$ and so the result follows from \cite[Theorem 3.3]{Ro2}.
\end{proof}

Note that this result is not true in general for non-standard $\bfs$.
From now on, by standard sequence of segments we mean those satisfying the property given in
Definition~\ref{defn:order}(2).

\subsection{The category $\afHr$-${\sf mod}^{\Z}$,
integral multisegments and standard words.}
 Let $\afHr$-${\sf mod}^{\Z}$ denote the full subcategory of
finite dimensional left $\afHr$-modules on which
the eigenvalues of $X_1,\ldots,X_r$ are powers of $q^2$ (i.e., the weights are integral).
It is known (cf. \cite{Kl, Gr}) that, to understand irreducible $\afHr$-modules,
it suffices to study those irreducible modules in $\afHr$-${\sf mod}^{\Z}$.
More precisely, given a standard sequence $\bfs\in\mathscr C^t_r$ with
$t\geq 1$, and assume
$
\bfs=\bfc^{(1)}\vee\cdots\vee\bfc^{(p)}
$
in the sense of Definition~\ref{defn:order}(2).
Then it is known \cite[\S 6.2]{Gr} (cf. \cite[Lemma 6.1.2]{Kl}) that ${\rm ind}^{\afHr}_{\mc H_{\vtg}(\mu)}(V_{\bfc^{(1)}}\otimes\cdots\otimes V_{\bfc^{(p)}})$
is irreducible and hence, by Proposition~\ref{prop:Ro3}, the following holds
\begin{equation}\label{gen&int}
V_\bfs\cong{\rm ind}^{\afHr}_{\mc H_{\vtg}(\mu)}(V_{\bfc^{(1)}}\otimes\cdots\otimes V_{\bfc^{(p)}}),
\end{equation}
where $\mu=(\mu_1,\ldots,\mu_p)$ with $\mu_i=\sum_j|\sfs^{(i)}_j|$.
Therefore, it is reduced to study the irreducible modules $V_{\bfs}$ for a segment chain
$\bfs$ on a line $L_a$.

We now further reduce the line $L_a$ to the line $L_1$ so that only integral multisegments will be considered.

 For $a\in\C^*$, let $\sigma_a$ be the automorphism of $\afHr$ given by
\begin{equation}\label{sigma}
\sigma_a(T_i)=T_i,\quad \sigma_a(X_k)=aX_k
\end{equation}
for $1\leq i\leq r-1$ and $1\leq k\leq r$.
For an $\afHr$-module $V$, we can twist the action with $\sigma_a$ to get a
new module denoted by $V^{\sigma_a}$.
Given an arbitrary segment $\sfs=(zq^{2i},\ldots,zq^{2j})$ and $a\in \C^*$, set
\begin{equation}\label{s-sigma}
\sfs^{\sigma_a}=a\sfs=(azq^{2i},\ldots,azq^{2j}).
\end{equation}
Similarly we can define $\bfs^{\sigma_a}$ and $\mfs^{\sigma_a}$
for $\bfs\in\mathscr C_r^t$ and $\mfs\in\mathscr S_r$.
Then by \eqref{Is} we obtain $I^{\sigma_a}_{\bfs}\cong I_{\bfs^{\sigma_a}}$ and hence
\begin{equation}\label{Vs-sigma}
V_{\mfs}^{\sigma_a}\cong V_{\mfs^{\sigma_a}}.
\end{equation}
using Proposition~\ref{prop:Ro2} and the $\Ha$-module isomorphism $V^{\sigma_a}_{\bfs}|_{\Ha}\cong V_{\bfs}|_{\Ha}$.
Hence, to understand irreducible modules $V_{\bfs}$, it is enough to
consider those sequence of segments $\bfs=(\sfs_1,\sfs_2,\ldots,\sfs_t)\in\mathscr C_r^t$ with all $\sfs_i$ lying on $L_1$.
A segment lying on $L_1$ is called an {\it integral segment}.

Define $\mathscr C_r^{\Z}$ to be the set of {\it standard sequences of integral segments} of total length $r$. Thus, every element in $\mathscr C_r^{\Z}$ is necessarily a chain of integral segments. Define similarly
the set of all {\it integral multisegments} of total length $r$
$$\mathscr S^{\Z}_{r}=\{\bar\bfs\in\mathscr S_r\mid \text{every segment in }\bfs \text{ lies on }L_1\}.$$

For any multisegment $\mathfrak s\in \mathscr S^{\Z}_r$, there exists a unique chain $\bfs$ such that $\bfs\in\mathfrak s$. Thus, the map
\begin{equation}
\aligned
 \flat_1:\mathscr C_r^{\Z}&\longrightarrow \mathscr S_r^{\Z},\quad
 \bfs\longmapsto\bar\bfs\endaligned
\end{equation}
is a bijection.
  We now use this fact to describe integral multisegments in terms of
 standard words.

Consider the set $\mc W$ of nonempty words in the alphabet $\mathbb Z_\leq^2=\{{j\choose i}\mid i,j\in\mathbb Z, i\leq j\}$. We always identify such a word ${j_1\choose i_1}\cdots{j_t\choose i_t}$ as ${{j_1\ldots j_t}\choose{i_1\ldots i_t}}$ and define
$\bigl|{{j_1\ldots j_t}\choose{i_1\ldots i_t}}\bigr|:=\sum_{k=1}^t(j_k-i_k+1).$
  Let
 $$\mc W(r)=\{w={{j_1\ldots j_t}\choose{i_1\ldots i_t}}\in\mc W\mid i_k\leq j_k, 1\leq k\leq t,  |w|=r\}.$$
Clearly, each word in $\mc W(r)$ defines an integral multisegment.  A word
${{j_1\ldots j_t}\choose{i_1\ldots i_t}}$ in $\mc W(r)$ is called a {\it standard word} if $j_1\leq j_2\leq\ldots\leq j_{t}$ and $i_k\geq i_{k+1}\text{ whenever }j_k=j_{k+1}$.
 Let $\mc W^s(r)$ be the set of standard words in $\mc W(r)$.

Clearly, by definition, there is a bijection
\begin{equation}\label{flat2}
\flat_2:\mc W^s(r)\longrightarrow\mathscr C_r^{\Z}
\end{equation}
sending a standard word ${{j_1j_2\ldots j_t}\choose{i_1i_2\ldots i_t}}$ to a chain 
$\bfs=(\sfs_1,\sfs_2,\ldots,\sfs_t)$, where,  for $1\leq k\leq t$,
$\sfs_k=(q^{2i_k},q^{2(i_k+1)},\ldots,q^{2j_k})$. This induces a bijection
\begin{equation}\label{flat}
\aligned
\flat=\flat_1\circ\flat_2:\mc W^s(r)&\longrightarrow\mathscr S^{\Z}_r,\quad
w\longmapsto \overline{\flat_2(w)}.\endaligned
\end{equation}

We remark that the reduction to the integral case allows us to link irreducible modules $V_{\mathfrak s}$ for $\mathfrak s\in\mathscr S^{\Z}_r$ with irreducible modules over the Ariki-Koike algebras of integral type considered in \S3.5.

%

\section{Ariki-Koike algebras and Specht modules}
In this section, we shall recall the construction
of Specht modules for Ariki-Koike algebras (cf. \cite{DJM, DR2}), compute the action of Jucys--Murphy elements on Specht modules, and introduce Kleshchev multipartitions for Ariki--koike algebras of integral type.

\subsection{Basics on multipartitions.}\label{subsec:notation}
Let $\mu=(\mu_1,\mu_2,\ldots,\mu_\ell)$ be a composition of $r$.
If, in addition, $\mu_1\geq\mu_2\geq\cdots\geq\mu_\ell$,
then $\mu$ is called a partition of $r$ and write $\mu\vdash r$.
Denote by $\mc P(r)$ the set of partitions of $r$.
Given a partition $\la$, its Young diagram
is known to be the set of points $\{(i,j)|1\leq i\leq\ell(\la), 1\leq j\leq\la_i\}
\subseteq\N^2$ and the elements $(i,j)$ are called {\em nodes}. Here $\ell(\la)$ denotes the number of nonzero parts of $\la$. The conjugate of $\la$ is the partition $\la'$ associated with the transpose of the Young diagram of $\la$.
An $m$-tuple of partitions $\und{\la}=(\la^{(1)},\la^{(2)},\ldots,\la^{(m)})$
with $|\la^{(1)}|+\cdots+|\la^{(m)}|=r$ is called
a multipartition of $r$. Call $\udla$ a {\it sincere} multipartition if $|\la^{(i)}|>0$ for all $i$.
Denote by $\mc P_m(r)$ the set of multipartitions
$\und{\la}=(\la^{(1)},\la^{(2)},\ldots,\la^{(m)})$ of $r$.
A multipartition $\und{\la}=(\la^{(1)},\ldots,\la^{(m)})$
can also be identified with its Young diagram
which is formally defined as the set
$\{(i,j,k)|1\leq i\leq\ell(\la^{(k)}), 1\leq j\leq\la_i^{(k)}, 1\leq k\leq m\}
\subseteq\N^3.$ The elements $(i,j,k)$ are called nodes.
Similarly, a $\und{\la}$-tableau is a labeling of the nodes in the diagrams
with numbers $1,2,\ldots, r$.
Let $t^{\und{\la}}$ (resp. $t_{\und{\la}}$) be the $\und{\la}$-tableau in
which the numbers $1,2,\ldots, r$ appear
in order along successive rows (resp. columns) in the first (resp. last) diagram of $\und{\la}$,
then in the second (resp. second last) diagram and so on.
For example, in the case $\und{\la}=((3,1), (2,2), (1))$,
we have
\begin{equation} \label{tableau}
\aligned
t^{\und{\la}}=\Big(~\young(123,4),\quad\young(56,78),\quad\young(9)~\Big),\quad
t_{\und{\la}}=\Big(~\young(689,7),\quad\young(24,35),\quad\young(1)~\Big).
\endaligned
\end{equation}

Given $\und{\la}\in\mc P_m(r)$, let
$$
\und{\la}'=(\la^{(m)'},\ldots,\la^{(2)'},\la^{(1)'})\quad (\text{resp.}, \und{\la}^{\vee}=\la^{(1)}\vee\la^{(2)}\vee\cdots\vee\la^{(m)})
$$
be its conjugate (resp.,  the composition of $r$ obtained by concatenating the components of $\und{\la}$).
Associated to $\und\la$, define the elements $w_{\und{\la}},w_{[\und{\la}]}\in\mf S_r$ by
\begin{equation} \label{lem:wudla}w_{\und{\la}} t^{\und{\la}}=t_{\und{\la}},\quad
w_{[\udla]}(a_{j-1}+b)
=r-a_j+b
\end{equation}
where $a_0=0, a_k=\sum^k_{j=1}|\la^{(j)}|$ for $1\leq k\leq m$, $1\leq b\leq |\la^{(j)}|$ and $
[\und{\la}]=[a_0,a_1,\ldots,a_m]
$.
Then we have
\begin{equation}\label{wudla-1}
w_{\und{\la}}^{-1}=w_{\und{\la}'}\text{ and } w_{\und{\la}}=w^{(1)}\cdots w^{(m)}w_{[\und{\lambda}]}
\end{equation}
where $w^{(i)}$ is the permutation sending the $i$th tableau of $w_{[\lambda]}t^{\und{\la}}$ to the $i$th tableau of $t_{\und{\la}}$; see \cite[(1.4)]{DR2}.

\subsection{\bf The Ariki-Koike algebra $\AKa$ and Jucys-Murphy elements.}
Recall $q\in\C^*$.
Let $\und{u}=(u_1, u_2,\ldots, u_m)\in\C^m$ with $m\geq 1$.
The {\em Ariki-Koike} algebra $\AKa$ is
the associative algebra over $\C$ generated by
$T_0, T_1,\ldots, T_{r-1}$ subject to the relations:
\begin{align*}
(T_i-q^2)(T_i+1)&=0,\quad 1\leq i\leq r-1, \\
T_iT_{i+1}T_i&=T_{i+1}T_iT_{i+1},\quad 1\leq i\leq r-2,\\
\quad T_iT_j&=T_jT_i,\quad 1\leq i,j\leq r-1,  |i-j|> 1, \\
(T_0-u_1)(T_0-u_2)&\cdots (T_0-u_m)=0,\\
T_0T_1T_0T_1&=T_1T_0T_1T_0.
\end{align*}
The {\em Jucys-Murphy} elements $L_k$ $(1\leq k\leq r)$ of the Ariki-Koike algebra $\AKa$
are defined by
\begin{equation}\label{JM2}
L_1=T_0,\quad L_k=q^{2(1-k)}T_{k-1}\cdots T_1T_0T_1\cdots T_{k-1}=q^{-2}T_{k-1}L_{k-1}T_{k-1}
\end{equation}
for $2\leq k\leq r$.
In the case $m=1$ and $u_1=1$, we write
$$J_k=L_k=q^{2(1-k)}T_{k-1}\cdots T_1T_1\cdots T_{k-1},\quad1\leq k\leq r.$$
It is easy to check that the Jucys-Murphy elements commute with each other.

If $u_1,\ldots,u_m$ are all non-zero,\footnote{This guarantees that all $L_k$ are invertible since all $X_k$ are invertible in $\afHr$.}
there exists a surjective homomorphism
\begin{equation}\label{pi}
\varpi_{\udu}: \afHr\twoheadrightarrow\AKa,\quad
X_k\mapsto L_k,\quad T_i\mapsto T_i
\end{equation}
for $1\leq k\leq r$ and $1\leq i\leq r-1$ and
$\varpi_{\udu}$ induces the algebra isomorphism
$$
\afHr/\langle (X_1-u_1)(X_1-u_2)\cdots(X_1-u_m)\rangle\overset\sim\longrightarrow \AKa.
$$
This means that any $\AKa$-module $M$ can be {\it inflated } to an $\afHr$-module in this case. We will use the same letter to denote its inflation.

A direct check of relations shows that there exists an anti-automorphism
\begin{equation} \label{tau}
\tau:\AKa\longrightarrow\AKa,
\end{equation}
defined by sending $T_i$ to $T_i$ for $0\leq i\leq r-1$.

\subsection{\bf Specht modules for $\AKa$.}
We shall mainly follow the constructions of Specht modules
for the Hecke algebra $\Ha$
and the Ariki-Koike algebras $\AKa$ established in \cite{DJ1,DJ2,DR1,DR2},
where right modules are considered.

Given a composition $\mu=(\mu_1,\mu_2,\ldots,\mu_{\ell})$ of $r$, set
$$
x_{\mu}:=\sum_{\sigma\in \mf S_{\mu}}T_{\sigma},\quad
y_{\mu}:=\sum_{\sigma\in\mf S_\mu}(-q^2)^{l(\sigma)}T_{\sigma}.
$$
It is easy to check that
\begin{align}\label{xylafin}
T_\sigma x_\mu=x_{\mu}T_\sigma=q^{2l(\sigma)}x_{\mu},\quad
T_\sigma y_\mu=y_{\mu}T_\sigma=(-1)^{l(\sigma)}y_{\mu}
\end{align}
for $\sigma\in\mf S_\mu$.
For $\la\in\mc P(r)$, define $z_{\la}:=y_{\la'}T_{w_{\la'}}x_{\la}$ and the Specht module $S^\la=\sH(r)z_\la$.

Following \cite{DR2}, we shall introduce the notion of Specht modules
for the Ariki-Koike algebras $\AKa$ in the rest of this section.
For $u\in\C$ and a positive integer $k$, let
$$
\pi_0(u)=1,\quad \pi_k(u)=(L_1-u)(L_2-u)\cdots(L_k-u).
$$
Writing $[\und{\lambda}]=[a_0,a_1,\ldots,a_m]$ with $a_0=0$ and $a_k=\sum^k_{j=1}|\lambda^{(j)}|$ with $1\leq k\leq m$ for a multipartition
$\und{\lambda}=(\lambda^{(1)},\ldots,\lambda^{(m)})$ of $r$, we define (cf. \cite{DJM, DR1,GL})
\begin{align*}
\pi_{[\und{\la}]}=\pi_{a_1}(u_2)\cdots\pi_{a_{m-1}}&(u_m),\quad
\widetilde{\pi}_{[\und{\la}]}=\pi_{a_1}(u_{m-1})\cdots\pi_{a_{m-1}}(u_1),\\
v_{[\und{\la}]}&=\widetilde{\pi}_{[\und{\la}']}T_{w_{[\und{\la}']}}\pi_{[\und{\la}]}.
\end{align*}
For a multipartition $\und{\la}$ of $r$, let
$$\aligned
x_{\und{\la}}=\pi_{[\und{\la}]}x_{\und{\la}^\vee}=x_{\und{\la}^{\!\vee}}\,\pi_{[\und{\la}]},\quad
y_{\und{\la}}&=\widetilde{\pi}_{[\und{\la}]}y_{\und{\la}^\vee}=y_{\und{\la}^\vee}\widetilde{\pi}_{[\und{\la}]},\quad
z_{\und{\la}}=y_{\und{\la}'}T_{w_{\und{\la}'}}x_{\und{\la}},\\
S^{\und{\la}}&=\AKa z_{\und{\la}}.
\endaligned
$$
The left ideal $S^{\und{\la}}$ of $\AKa$ generated by $z_{\und{\la}}$ is called the {\it cyclotomic Specht module}  associated to a multipartition $\und{\la}\in\mc P_m(r)$. By~\eqref{xylafin}, the following holds for $\sigma\in\mf S_{(\udla')^\vee}$:
\begin{equation}\label{Tzudla}
T_{\sigma}z_{\udla}=(-1)^{l(\sigma)}z_{\udla}.
\end{equation}




We remark that, by \cite[Theorem 2.9]{DR2} , the Specht module $S^{\und{\la}}$ defined above
is isomorphic to the cell modules associated with $\udla'$ defined in \cite{DJM}.
The following lemma will be used later on.
\begin{lem}[{\cite[Corollary 2.3]{DR2}}]\label{lem:Specht}
Suppose that $\und{\la}=(\la^{(1)},\ldots,\la^{(m)})$ is a multipartition of $r$
with  $|\la^{(k)}|=r_k$ for $1\leq k\leq m$.
Then
\begin{equation}\label{zv}
z_{\und{\la}}=v_{[\und{\la}]}(z_{\la^{(1)}}\otimes\cdots \otimes z_{\la^{(m)}})
=(z_{\la^{(m)}}\otimes\cdots \otimes z_{\la^{(1)}})v_{[\und{\la}]},
\end{equation}
where $z_{\la^{(1)}}\otimes\cdots \otimes z_{\la^{(m)}}
\in\mc H(r_1)\otimes\cdots\otimes\mc H(r_m)\subseteq\Ha$.
Moreover, as $\Ha$-modules, the following holds
\begin{equation}\label{slaind}
S^{\und{\la}}|_{\Ha}\cong {\rm ind}^{\Ha}_{\mc H(r_m)\otimes\cdots\otimes\mc H(r_1)}
(S^{\la^{(m)}}\otimes \cdots\otimes S^{\la^{(1)}}).
\end{equation}
\end{lem}


\subsection{\bf The action of Jucys-Murphy elements on $S^{\udla}$.}  For later use, we now study the action of Jucys-Murphy elements on cyclotomic
Specht modules.

Given $\la\in\mc P(r)$ and a node $\rho=(i,j)\in\la$,  define the residue ${\rm res}(\rho)$ of $\rho$ relative to $q$
by ${\rm res}(\rho):=q^{2(j-i)}.$ Accordingly, for a $\la$-tableau $t$ and $1\leq k\leq r$,
let ${\rm res}_t(\rho)$ be the residue of the node occupied by
$k$ in $t$. Generally, for $\und{\la}=(\la^{(1)},\la^{(2)},\ldots,\la^{(m)})\in\mc P_m(r)$
and $\rho=(i,j,k)\in\und{\la}$, define its residue relative to
the given parameters $\{q,u_1,\ldots,u_m\}$
by
$$
{\rm res}(\rho):=u_kq^{2(j-i)}.
$$
For a $\und{\la}$-tableau $t=(t^{(1)},t^{(2)},\ldots,t^{(m)})$ and $1\leq k\leq r$, let
${\rm res}_t(k)$ be the residue of the node occupied by $k$ in $t$.
For example, if $t=t_{\udla}$ as given in \eqref{tableau}, then ${\rm res}_t(4)=u_2q^2$.
\begin{prop}\label{prop:JMact}
The following holds for any multipartition
$\und{\la}=(\la^{(1)},\ldots,\la^{(m)})$ of $r$:
$$
L_kz_{\und{\la}}
={\rm res}_{t_{\und{\la}}}(k)z_{\und{\la}}, \quad 1\leq k\leq r.
$$
\end{prop}
\begin{proof}
Fix $1\leq k\leq r$ and assume $[\udla]=\bfa=[a_0,a_1,\ldots,a_m]$.
Suppose that $k=r-a_j+b=|\la^{(j+1)}|+\cdots+|\la^{(m)}|+b$
for $1\leq j\leq m$ with $1\leq b\leq |\la^{(j)}|$.
Then note that the node in $t_{\udla}$ occupied by $k$
coincides with the node occupied by $b$ in $t_{\la^{(j)}}$, where
$t_{\udla}=(t_{\la^{(1)}},\ldots,t_{\la^{(m)}})$. Hence,
\begin{equation}\label{tworesidue}
{\rm res}_{t_{\udla}}(k)=u_j{\rm res}_{t_{\la^{(j)}}}(b).
\end{equation}
Observe that by~\eqref{JM2} the following holds
$$
L_k=L_{r-a_j+b}=q^{2(1-b)}T_{r-a_j+b-1}\cdots T_{r-a_j+1}L_{r-a_j+1}T_{r-a_j+1}\cdots T_{r-a_j+b-1}.
$$
It follows from \cite[Proposition 3.1]{DR1} that we have
\begin{align*}
L_{r-a_j+1}v_{[\und{\la}]}
=u_jv_{[\und{\la}]},\quad
T_{r-a_j+b}v_{[\und{\la}]}
=v_{[\und{\la}]}T_{a_{j-1}+b}.
\end{align*}
Then by~\eqref{zv}, a direct calculation shows
\begin{align}
L_kz_{\und{\la}}
=&q^{2(1-b)}T_{r-a_j+b-1}\cdots T_{r-a_j+1}L_{r-a_j+1}T_{r-a_j+1}\cdots T_{r-a_j+b-1}z_{\und{\la}}\notag\\
=&q^{2(1-b)}T_{r-a_j+b-1}\cdots T_{r-a_j+1}L_{r-a_j+1}T_{r-a_j+1}\cdots T_{r-a_j+b-1}
v_{[\und{\la}]}(z_{\la^{(1)}}\otimes\cdots \otimes z_{\la^{(m)}})\notag\\
=&u_jq^{2(1-b)}v_{[\und{\la}]}T_{a_{j-1}+b-1}\cdots T_{a_{j-1}+1}
T_{a_{j-1}+1}\cdots T_{a_{j-1}+b-1}(z_{\la^{(1)}}\otimes\cdots \otimes z_{\la^{(m)}}).
 \label{zL1}
\end{align}
Meanwhile, observe that the element
$q^{1-b}T_{a_{j-1}+b-1}\cdots T_{a_{j-1}+1}T_{a_{j-1}+1}\cdots T_{a_{j-1}+b-1}$ belongs to
the subalgebra $\mc H(\mu)$ of $\mc H(r)$, where $\mu=(|\la^{(1)}|,\ldots,|\la^{m}|)$,  and
can be identified with $1^{\otimes j-1}\otimes J_b\otimes1^{\otimes m-j}$ under the first isomorphism in \eqref{two iso}.
Then it follows from \cite[Theorem 3.14]{DJ2} that $J_b z_{\la^{(j)}}={\rm res}_{t_{\la^{(j)}}}(b) z_{\la^{(j)}}$. Hence,
\begin{align*}
&q^{2(1-b)}T_{a_{j-1}+b-1}\cdots T_{a_{j-1}+1}T_{a_{j-1}+1}\cdots T_{a_{j-1}+b-1}
(z_{\la^{(1)}}\otimes\cdots \otimes z_{\la^{(m)}})
\\
&=z_{\la^{(1)}}\otimes\cdots \otimes J_b z_{\la^{(j)}} \otimes\cdots\otimes z_{\la^{(m)}}
={\rm res}_{t_{\la^{(j)}}}(b)(z_{\la^{(1)}}\otimes\cdots\otimes z_{\la^{(j)}} \otimes\cdots \otimes z_{\la^{(m)}}).
\end{align*}
Thus, by \eqref{zL1} and \eqref{tworesidue}, one can deduce that
$$
L_k z_{\und{\la}}=u_j{\rm res}_{t_{\la^{(j)}}}(b)
v_{[\und{\la}]}(z_{\la^{(1)}}\otimes\cdots \otimes z_{\la^{(m)}})
={\rm res}_{t_{\und{\la}}}(k)z_{\und{\la}},
$$
as desired.
\end{proof}

The following result will be used later on.
\begin{prop} \label{prop:induceSpecht}
Suppose that $\udu=(u_1,\ldots,u_m)\in\mathbb (\C^*)^m$ and
$\udla=(\la^{(1)},\ldots,\la^{(m)})\in\mc P_m(r)$.
Then the inflation of the $\AKa$-module $S^{\udla}$
gives an $\afHr$-module isomorphism:
\begin{align*}
S^{\udla}\cong {\rm ind}^{\afHr}_{\mc H_{\vtg}(\mu)}
({\rm ev}^*_{u_m}(S^{\la^{(m)}})\otimes\cdots\otimes{\rm ev}^*_{u_1}(S^{\la^{(1)}})),
\end{align*}
where $\mu=(|\la^{(m)}|,\ldots,|\la^{(2)}|,|\la^{(1)}|)$.
\end{prop}
\begin{proof}
Recall that $\mc H_{\vtg}(\mu)=\langle\sH(\mu),X^{\pm1}_1,\ldots, X^{\pm1}_r\rangle$.
By~\eqref{pi} and Proposition~\ref{prop:JMact},
it is straightforward to check that the following map
\begin{align*}
\zeta: {\rm ev}^*_{u_m}(S^{\la^{(m)}})\otimes\cdots\otimes{\rm ev}^*_{u_1}(S^{\la^{(1)}})
&\longrightarrow S^{\udla},\quad
z_{\la^{(m)}}\otimes\cdots\otimes z_{\la^{(1)}}\longmapsto z_{\udla}
\end{align*}
is an $\mc H_{\vtg}(\mu)$-homomorphism.
Then, by Frobenius reciprocity, there exists a nonzero homomorphism
$\tilde{\zeta}:{\rm ind}^{\afHr}_{\mc H_{\vtg}(\mu)}
{\rm ev}^*_{u_m}(S^{\la^{(m)}})\otimes\cdots\otimes{\rm ev}^*_{u_1}(S^{\la^{(1)}})
\rightarrow S^{\udla}$ which is also surjective since $S^{\udla}$ is generated by $z_{\udla}$.
By~\eqref{slaind}, a dimension comparison proves the proposition.
\end{proof}


%
\subsection{\bf Ariki--Koike algebras of integral type and Kleshchev multipartitions.}
The classification of irreducible $\AKa$-modules has been completed by Ariki using Kleshchev multipartitions;
see \cite[Theorem 4.2]{Ar} and \cite[conjecture 2.12]{AM}.
It is also known from \cite[\S1]{AM} that the classification of irreducible
$\AKa$-modules is reduced to the cases where the parameters $u_1,\ldots, u_m$
are either all zero or all powers of $q^2$. For our purpose, we are mainly concerned about the latter one
(see footnote 2).

For $m\geq 1$, let
$$
\mathfrak F_m=\{(f_1,\ldots,f_m)\in\Z^m\mid f_1\geq\cdots\geq f_m\},
$$
and let
\begin{equation}\label{set F}
\mathfrak F=\bigcup_{m\geq1}\mathfrak F_m.
\end{equation}
For $f=(f_1,\ldots,f_m)\in\mathfrak{F}$, let $f^*=(-f_m,\ldots,-f_1)$.  Clearly, this gives a bijection
 $$(\ )^*:\mathfrak F\longrightarrow\mathfrak F.$$

Recall that $q$ is {\it not a root of unity} and consider the parameter vectors of the form
\begin{equation}\label{parameters for integral AKa}
\udu_f=(q^{2f_1},\ldots,q^{2f_m}) \quad (f\in\mathfrak F_m),
\end{equation}
That is, $u_1=(q^2)^{f_1},\ldots, u_m=(q^2)^{f_m}$. In addition, we set $\udu_f^{-1}=\udu_{f^*}=(q^{-2f_m},\ldots,q^{-2f_1})$.
An Ariki--Koike algebra $\AKa$ is of {\it integral type} if $\udu=\udu_f$ for some $f\in\mathfrak F$.
Thus, Kleshchev multipartions for an Ariki--Koike algebra $\AKa$ of integral type (and a parameter which is not a root of unity) can be described as follows (cf. \cite{Va}).

\begin{defn}\label{KMP}
 Given a multipartition $\la=(\la^{(1)},\ldots,\la^{(m)})$ of $r$,
we will say that $\udla$ is a
{\em Kleshchev multipartition} with respect to $f\in\mathfrak F_m$ (or  to $\udu_f$) if it satisfies
\begin{equation}\label{Kleshchev}
\la^{(k)}_{j+f_k-f_{k+1}}\leq\la^{(k+1)}_j
\end{equation}
for $j\geq 1$ and $1\leq k\leq m-1$.

Denote by $\mc K_{\udu_f}(r)$ or $\mc K_f(r)$ the set of Kleshchev multipartitions of $r$
with respect to $\udu_f$.
\end{defn}

Note that, when $q$ is a root of unity, the definition of Kleshchev multipartitions is much more complicated.

\begin{rem}\label{reduce to 1} If $\ell(\la^{(k)})\leq f_k-f_{k+1}+1$, then \eqref{Kleshchev} is equivalent to $\la^{(k)}_{1+f_k-f_{k+1}}\leq\la^{(k+1)}_1$.
\end{rem}

For $\udu=(u_1,\ldots,u_m)$, set
$$
\udu^{-1}=(u_m^{-1},\ldots,u_1^{-1}).
$$
Then by~\eqref{Kleshchev}, a multipartition $\und{\ga}=(\ga^{(1)},\ldots,\ga^{(m)})$ belongs to $\mc K_{\udu_f^{-1}}(r)=\mc K_{f^*}(r)$
if and only if
\begin{equation}\label{Kinv}
\ga^{(k)}_{j+f_{m-k}-f_{m-k+1}}\leq\ga^{(k+1)}_j
\end{equation}
for $j\geq 1$ and $1\leq k\leq m-1$.

 It is known \cite[Corollary 1.3]{BK1} that the Ariki-Koike algebra $\mc H_{\udu_f}(r)$
for $f\in\mathfrak F$ and $q$ not root of unity is isomorphic to the corresponding degenerate cyclotomic Hecke algebra
over $\C$. Then we have the following.
\begin{prop}[{\cite[Theorem 3.8, Lemma 3.13]{BK} (cf. \cite{AM, Ar, Va})}]
\label{prop:BK}
Assume that $f=(f_1,\ldots,f_m)\in\mathfrak F$ with $m\geq 1$.
The following holds for all $r\geq 1$:
\begin{enumerate}
\item For multipartition $\udla=(\la^{(1)},\ldots,\la^{(m)})\in\mc P_m(r)$ satisfying $\udla'\in\mc K_{f^*}(r)$,
the $\afHr$-module
${\rm ind}^{\afHr}_{\mc H_{\vtg}(\mu)}
({\rm ev}^*_{q^{2f_m}}(S^{\la^{(m)}})\otimes\cdots\otimes{\rm ev}^*_{q^{2f_1}}(S^{\la^{(1)}}))$
has irreducible head denoted by $D^{\udla}$, where $\mu=(|\la^{(m)}|,\ldots,|\la^{(1)}|)$.
In particular, $D^{\udla}$ affords an irreducible $\mc H_{\udu_f}(r)$-module.

\item The set $\{D^{\udla'}~| \udla\in\mc K_{f^*}(r)\}$ is a complete set of non-isomorphic
irreducible $\mc H_{\udu_f}(r)$-modules.

\end{enumerate}
\end{prop}

By Proposition~\ref{prop:induceSpecht} and Proposition~\ref{prop:BK}, we have the following.
\begin{cor}\label{cor:hdSpecht}
For $\udla\in\mc P_m(r)$ satisfying $\udla'\in\mc K_{f^*}(r)$ with $f\in\mathfrak F_m$, the cyclotomic Specht module
$S^{\udla}$ over the algebra $\mc H_{\udu_f}(r)$  has irreducible head isomorphic to $D^{\udla}$.
\end{cor}

We will construct in the next section a subset of $\cup_{f\in\mathfrak F}\mc K_f(r)$ which labels precisely the irreducible $\afHr$-modules
with ``integral weights'', i.e., the irreducible objects in $\afHr$-${\sf mod}^{\Z}$.


\section{standard multipartitions and integral multisegments}

At the end of \S 2, we saw that irreducible objects in the category $\afHr$-${\sf mod}^{\Z}$ are indexed by the set $\mathscr S_r^\Z$. The set can be further described in terms of standard sequences of integral segments or standard words. In this section,
we will describe the set in terms of standard multipartitions. In particular, we will characterise the irreducible objects in the category $\afHr$-${\sf mod}^{\Z}$ by these standard multipartitions and their associated irreducible modules of Ariki--Koike algebras.

\subsection{Column residual segments of a multipartition.}
We first construct a sequence $\bfs$ of segments via the column residual segments of  a multipartition $\udla$ and show that $I_\bfs$ maps onto $S^{\udla}$.

For $\und{\la}=(\la^{(1)},\la^{(2)},\ldots,\la^{(m)})\in\mc P_m(r)$ and parameter vector $\udu=(u_1,\ldots,u_m)$,
define the $j$th-{\it column residual segment} of $\la^{(i)}$
$$
\sfs^{(i)}_j=(u_iq^{2(j-(\la^{(i)'})_j)},\ldots, u_iq^{2(j-2)},u_iq^{2(j-1)})\quad (1\leq i\leq m, 1\leq j\leq\la^{(i)}_1)
$$
by reading the residues of the nodes in the $j$-th column of $\la^{(i)}$ from bottom to top, and form
the sequence of column residual segments from column 1, then column 2, and so on, starting from the last partition of $\udla$, then the second last partition, and so on,
\begin{equation}\label{scla}
\bfs^{\sfc}_{\und{\la};\udu}=\big(\sfs_1^{(m)},\ldots,\sfs^{(m)}_{\la^{(m)}_1},\ldots,
\sfs^{(1)}_1,\ldots,\sfs^{(1)}_{\la^{(1)}_1}\big);
\end{equation}
Compare the order of the Young diagram $t_{\udla}$ in \eqref{tableau}.
 We will simply write $\bfs^{\sfc}_{\und{\la};f}$  in the sequel, if $\udu=\udu_f$ for some $f\in \mathfrak F_m$.
\begin{lem}\label{lem:Ibfs}
Assume $\und{\la}=(\la^{(1)},\la^{(2)},\ldots,\la^{(m)})\in\mc P_m(r)$ and
$\udu=(u_1\ldots,u_m)\in(\C^*)^m$.
Write $\bfs=\bfs^{\sfc}_{\und{\la};\udu}$.
Then there exists a surjective $\afHr$-homomorphism
$I_{\bfs}
\twoheadrightarrow S^{\und{\la}}$.
\end{lem}
\begin{proof}
By~\eqref{Tzudla} and Proposition~\ref{prop:JMact}, one can deduce that
$$
T_i z_{\und{\la}}=-z_{\und{\la}},\quad
X_k z_{\und{\la}}=L_k z_{\und{\la}}={\rm res}_{t_{\und{\la}}}(k)z_{\und{\la}}
$$
for any $T_i\in\mc H((\und{\la}')^\vee)$ and $1\leq k\leq r$.
This means the space $\C z_{\und{\la}}$
affords an $\mc H_{\vtg}((\und{\la}')^\vee)$-submodule of the inflation of the $\AKa$-module
$S^{\udla}$.
Then by Frobenius reciprocity, there exists a surjective $\afHr$-homomorphism
$${\rm Ind}^{\afHr}_{\mc H_{\vtg}((\und{\la}')^\vee)}\C z_{\und{\la}}
\twoheadrightarrow S^{\und{\la}}.$$

Observe from~\eqref{scla}, \eqref{tilde s}, and the standard tableau $t_{\udla}$ (see \eqref{tableau}) that the following holds (see \eqref{tilde s})
\begin{equation}
\widetilde{\bfs}^\vee
=({\rm res}_{t_{\und{\la}}}(1),{\rm res}_{t_{\und{\la}}}(2),
\ldots,{\rm res}_{t_{\und{\la}}}(r)),\quad
\mu(\bfs)
=(\und{\la}')^\vee.\label{mubfs}
\end{equation}
Then it is straightforward to check that $\C z_{\und{\la}}\cong\C_{\bfs}$ as $\mc H_{\vtg}((\und{\la}')^\vee)$-modules.
Hence the lemma follows from \eqref{Is}.
\end{proof}
\subsection{\bf Standard multipartitions associated with integral multisegments.}
We shall construct a Kleshchev multipartition associated with an integral multisegment as follows.

Recall the sets  $\mathfrak F_m,\mathfrak F$ in \eqref{set F} and the elements
$\udu_f=(q^{2f_1},q^{2f_2},\ldots,q^{2f_m})$ for $f=(f_1,\ldots,f_m)\in\mathfrak F_m$
in \eqref{parameters for integral AKa}.
\begin{defn}\label{SKM}
A multipartition $\udga=(\ga^{(1)},\ldots,\ga^{(m)})$ of $r$
is called a {\it standard multipartition} relative to $f=(f_1,\ldots,f_m)\in\mathfrak F$, if it satisfies
\begin{itemize}
\item[(SK1)] $\udga$ is sincere (or equivalently, $\ell(\ga^{(a)})\geq 1$ for each $ 1\leq a\leq m$).

\item[(SK2)]$\ell(\ga^{(a)})\leq f_a-f_{a+1}+1$ and $\ga^{(a)}_{f_a-f_{a+1}+1}\leq\ga^{(a+1)}_1$
for each $ 1\leq a\leq m-1$.

\item[(SK3)] If $\ell(\ga^{(a)})= f_a-f_{a+1}$ for some $1\leq a\leq m-1$,
then $\ga^{(a)}_{f_a-f_{a+1}}\leq \ga^{(a+1)}_1-1$.
\end{itemize}
\end{defn}
Observe from the condition (SK2) that a standard multipartition is a Kleshchev multipartition; see Remark \ref{reduce to 1}.

Let $\mc K^s_{f}(r)$ be the set of all standard multipartitions of $r$ relative to $f\in\mathfrak F$. Then
$$
\mc K^s_{f}(r)\subseteq\mc K_{f}(r)=\mc K_{\udu_f}(r).
$$







Given ${{j_1j_2\ldots j_t}\choose{i_1i_2\ldots i_t}}\in\mc W^s(r)$ so that
 $\mathfrak s=\flat{{j_1j_2\ldots j_t}\choose{i_1i_2\ldots i_t}}\in\mathscr S^{\Z}_{r}$ (see \eqref{flat}),
let $1\leq k_1\leq t$ be the right-most column index satisfying $j_a=j_{a-1}+1$, for $2\leq a\leq k_1$, and
$i_1<i_2<\cdots<i_{k_1}$.
Define $\ga^{(1)}=(\ga^{(1)}_1,\ldots,\ga^{(1)}_{k_1})$ by setting
$\ga^{(1)}_b=j_b-i_b+1$ for $1\leq b\leq k_1$. Note that $k_1=j_{k_1}-j_1+1$ and $\ga^{(1)}$ is a partition.

Let
$$
\mathfrak s^1=\flat{{j_{k_1+1}j_{k_1+2}\ldots j_t}\choose{i_{k_1+1}i_{k_1+2}\ldots i_t}} \in\mathscr S^{\Z}_{r-|\ga^{(1)}|}.
$$
Now applying the same procedure to $\mathfrak s^1$, we obtain
a positive number $1\leq k_2\leq t-k_1$, a partition $\ga^{(2)}$ and $\mathfrak s^2\in\mathscr S^{\Z}_{r-|\ga^{(1)}|-|\ga^{(2)}|}$.
Continuing in this way, we will end up with $1\leq k_1,\ldots, k_m\leq t$ and a multipartition
\begin{equation}\label{ga_s}
\underline{\ga}_\fks=(\ga^{(1)},\ga^{(2)},\ldots,\ga^{(m)}),
\end{equation}
where $ \ga^{(a)}=(\ga^{(a)}_1,\ga^{(a)}_2,\ldots,\ga^{(a)}_{k_a}) $ with
$
\ga^{(a)}_b=j_{k_1+\cdots+k_{a-1}+b}-i_{k_1+\cdots+k_{a-1}+b}+1
$
for all $1\leq a\leq m$ and $1\leq b\leq k_a$. Note that
\begin{equation}\label{ka}
k_a=j_{k_1+\cdots+k_a}-j_{k_1+\cdots+k_{a-1}+1}+1.
\end{equation}
Meanwhile, define
\begin{equation}\label{us}
\underline{u}_{\mathfrak s}=(u_1,\ldots,u_m)
\end{equation}
by setting
$$ u_1=q^{2(j_{k_1+\cdots+k_{m-1}+1})},~
u_2=q^{2(j_{k_1+\cdots+k_{m-2}+1})},~
\ldots,
u_m=q^{2j_1}.$$
(So $f_i=j_{k_1+\cdots+k_{m-i}+1}$ in the notation of \eqref{parameters for integral AKa}.)
Observe that, for $\mfs,\mft\in\mathscr S^{\Z}_{r}$,
\begin{equation}\label{equiv}
\mfs=\mft \text{ if and only if }  \udga_{\mfs}=\udga_{\mft} , \udu_{\mfs}=\udu_{\mft}.
\end{equation}
For every $f\in\mathfrak F$, recall the element $\udu_f$ defined in \eqref{parameters for integral AKa}. Let
$$\mathscr S_{r,f}^\Z:=\{\mathfrak s\in\mathscr S_r^\Z~|~\udu_{\mathfrak s}=\udu_f\}.$$

\begin{Example}
Suppose
$$\mathfrak s=\flat(w),\text{ where }w=\biggl(\begin{matrix}-1&0&1&2&2\\
-4&-5&-2&-1&-2\\\end{matrix}\biggr).$$
Then by the procedure described above, one can deduce that $k_1=1,k_2=3, k_3=1$, and
$\und\ga_\fks=(\ga^{(1)},\ga^{(2)},\ga^{(3)})$, where
$
\ga^{(1)}=(4),
\ga^{(2)}=(6,4,4),\text{ and }
\ga^{(3)}=(5).$
Moreover, $\udu_\fks=(u_1,u_2,u_3)$ with
$
u_1=q^4,u_2=q^0,u_3=q^{-2}.
$
We also observe from this example that $\und\ga_\fks\in\mc K_{\udu_\fks^{-1}}(r)$ is a standard multipartition. Moreover, for multipartition
$$
\udla_\fks=\und\ga_\fks'=((1^5),(3^41^2),(1^4))
$$ and the $u_\fks$ above,
 the multisegment $\bfs=\bfs^{\sfc}_{\udla_\fks;\udu_\fks}$ constructed in \eqref{scla}
 is the standard element in $\fks=\bar\bfs$.
\end{Example}
In fact, the last observation holds in general.

\begin{prop}\label{prop:stoK}
For $\fks\in\mathscr S^\Z_r$, let $\ga_\fks$ and $u_\fks$ be defined as in \eqref{ga_s} and \eqref{us}.
Then:-
\begin{enumerate}
\item $\underline{\ga}_\fks$ is a sincere standard multipartition with respect to
$u_\fks^{-1}$ and, moreover, $\bfs^{\sfc}_{\udga_\fks';\udu_\fks}$ is the standard element in $\fks$;
\item for a given $f\in\mathfrak F$, the map
\begin{align*}
\eta_{f}:\mathscr S_{r,f}^\Z
&\longrightarrow \mc K^s_{\udu_f^{-1}}(r)=\mc K^s_{f^*}(r),\quad
\mathfrak s\longmapsto \udga_{\mathfrak s}.
\end{align*}
is a bijection.
\end{enumerate}
\end{prop}
\begin{proof}(1)
Suppose $\mathfrak s\in \mathscr S_{r,f}^\Z$ and
assume that $\mathfrak s =\flat(w)$  with $w={{j_1j_2\ldots j_t}\choose{i_1i_2\ldots i_t}}\in\mc W^s(r)$. Then, $\udu_\fks=\udu_f$ and so $f_i=j_{k_1+\cdots+k_{m-i}+1}$ with $k_a$ as given in \eqref{ka}. In particular,
$$j_{k_1+\cdots+k_{a-1}+1}=f_{m-a+1}, \quad j_{k_1+\cdots+k_a}=f_{m-a+1}+k_a-1$$
for all $1\leq a\leq m$. Write $\underline{\ga}_\fks=(\ga^{(1)},\ga^{(2)},\ldots,\ga^{(m)})$
with $\ell(\ga^{(a)})=k_a=j_{k_1+\cdots+k_a}-j_{k_1+\cdots+k_{a-1}+1}+1$ for $1\leq a\leq m$. So $\udga_\fks$ is sincere.
In the following we shall show that $\udga_{\mathfrak s}$  satisfies the conditions (SK1)-(SK3).

First,
$
\ell(\ga^{(a)})=k_a\geq 1, \text{ for all }1\leq a\leq m,
$
and $f^*_a=-f_{m-a+1}=-j_{k_1+\cdots+k_{a-1}+1} $ for $1\leq a\leq m$.
Thus, $\udga_{\mathfrak s}$ satisfies
(SK1).
Also,  the hypothesis that the word $w$ is standard implies that
\begin{align*}
k_a-1+f_{m-a+1}=j_{k_1+\cdots+k_a}\leq j_{k_1+\cdots+k_a+1}=f_{m-a}.\quad
\end{align*}
Hence,
$
k_a\leq f_{m-a}-f_{m-a+1}+1=f^*_{a}-f^*_{a+1}+1
$
for all $1\leq a\leq m-1$. Hence to prove (SK2) holds, by~\eqref{Kinv} and~\eqref{us}, it suffices to show that
the following holds
for $1\leq a\leq m-1$ and $b\geq 1$:
\begin{equation}\label{inequality}
\ga^{(a)}_{b+j_{k_1+\cdots+k_a+1}-j_{k_1+\cdots+k_{a-1}+1}}\leq \ga^{(a+1)}_b
\end{equation}
Fix $1\leq a\leq m-1$ and $b\geq 1$.
If $\ga^{(a)}_{b+j_{k_1+\cdots+k_a+1}-j_{k_1+\cdots+k_{a-1}+1}}=0$, then
~\eqref{inequality} automatically holds.
Suppose now $\ga^{(a)}_{b+j_{k_1+\cdots+k_a+1}-j_{k_1+\cdots+k_{a-1}+1}}>0$.
Since $w$ is standard, the sequence $j_1,\ldots,j_t$ is weakly increasing. By \eqref{ka},
\begin{equation}\label{length ka}
j_{k_1+\cdots+k_a+1}-j_{k_1+\cdots+k_{a-1}+1}\geq
j_{k_1+\cdots+k_a}-j_{k_1+\cdots+k_{a-1}+1}=k_a-1.
\end{equation}
Hence, we obtain from the definition of the partition $\ga^{(a)}$ that
$$0<\ga^{(a)}_{b+j_{k_1+\cdots+k_a+1}-j_{k_1+\cdots+k_{a-1}+1}}\leq
\ga^{(a)}_{b+k_a-1},$$
which forces $b=1$ and $1+j_{k_1+\cdots+k_a+1}-j_{k_1+\cdots+k_{a-1}+1}=k_a$
since $\ell(\ga^{(a)})=k_a$ (cf. Remark \ref{reduce to 1}).
Thus, \eqref{length ka} must be an equality and consequently,
$
 j_{k_1+\cdots+k_a+1}=j_{k_1+\cdots+k_a}.
$
On the other hand, the fact that $w$ is standard implies $
i_{k_1+\cdots+k_a}\geq i_{k_1+\cdots+k_a+1}.$ Hence,
\begin{align*}
\ga^{(a)}_{b+j_{k_1+\cdots+k_a+1}-j_{k_1+\cdots+k_{a-1}+1}}
&=\ga^{(a)}_{k_a}=j_{k_1+\cdots+k_a}-i_{k_1+\cdots+k_a}+1\\
&\leq j_{k_1+\cdots+k_a+1}-i_{k_1+\cdots+k_a+1}+1
=\ga^{(a+1)}_1=\ga^{(a+1)}_b
\end{align*}
proving \eqref{inequality}.

It remains to prove (SK3) holds. If $k_a=f^*_{a}-f^*_{a+1}=f_{m-a}-f_{m-a+1}$ for some $1\leq a\leq m-1$, then
we have
$$
j_{k_1+\cdots+k_a}=f_{m-a+1}+k_a-1=f_{m-a}-1=j_{k_1+\cdots+k_a+1}-1.
$$
This implies, by the choice of $k_a$ in the construction of $\udga_{\mathfrak s}$, that
$
i_{k_1+\cdots+k_a}\geq i_{k_1+\cdots+k_a+1}.
$
Hence,
\begin{align*}
\ga^{(a)}_{k_a}
&=j_{k_1+\cdots+k_a}-i_{k_1+\cdots+k_a}+1\\
&\leq j_{k_1+\cdots+k_a+1}-i_{k_1+\cdots+k_a+1}=\ga^{(a+1)}_1-1,
\end{align*}
proving (SK3) and so $\udga_{\mathfrak s}\in\mc K^s_{f^*}(r)$. In other words, we have proved ${\rm im}(\eta_{f})\subseteq\mc K^s_{f^*}(r)$.
The last assertion follows from the definition.

(2)
By~\eqref{equiv}, the map $\eta_{f}$ is injective.
Let $\udga=(\ga^{(1)},\ga^{(2)},\ldots,\ga^{(m)})\in\mc K^s_{f^*}(r)$ be a standard multipartition (or satisfy (SK1)--(SK3)). We now construct an integral multisegment $\fks\in\mathscr S_r^\Z$ such that $\udu_\fks=\udu_f$ and $\udga_\fks=\udga.$

Set $k_a=\ell(\ga^{(a)})$ for $1\leq a\leq m$ and let $t=k_1+\cdots+k_m$.
Given $1\leq c\leq t$, assume that
$c=k_1+k_2+\cdots+k_{a-1}+b$ with $1\leq b\leq\ell(\ga^{(a)})=k_a$ for some $1\leq a\leq m$ and define
\begin{equation}\label{bijectionji}
j_c=f_{m-a+1}+b-1, \quad i_c=f_{m-a+1}-\ga^{(a)}_b+b.
\end{equation}
Clearly, $i_c\leq j_c$ for $1\leq c\leq t$ since $\ga^{(a)}_b\geq 1$.
We claim that
\begin{equation}\label{standard}
{{j_1j_2\ldots j_t}\choose{i_1i_2\ldots i_t}}\in\mc W^s(r).
\end{equation}
Indeed, firstly, since $\udga=(\ga^{(1)},\ga^{(2)},\ldots,\ga^{(m)})\in\mc K^s_{f^*}(r)$,
we have $k_a\leq f_{m-a}-f_{m-a+1}+1$ by (SK2). In other words,
$
f_{m-a+1}+k_a-1\leq f_{m-a},
$
which together with \eqref{bijectionji} implies that
\begin{equation}\label{bijectionjka}
j_{k_1+\cdots+k_a}\leq j_{k_1+\cdots+k_a+1}
\end{equation}
for $1\leq a\leq m-1$.
If $j_{k_1+\cdots+k_a}= j_{k_1+\cdots+k_a+1}$ for some $1\leq a\leq m-1$,
then $f_{m-a+1}+k_a-1=f_{m-a}$ and so $\ell(\ga^{(a)})=k_a=f_{m-a}-f_{m-a+1}+1$.
Then, by (SK2),
\begin{align*}
i_{k_1+\cdots+k_a}
=f_{m-a+1}-\ga^{(a)}_{k_a}+k_a
=f_{m-a}-\ga^{(a)}_{k_a}+1
\geq f_{m-a}-\ga^{(a+1)}_1+1=i_{k_1+\cdots+k_a+1}.
\end{align*}
Thus, this together with \eqref{bijectionji} and \eqref{bijectionjka} shows that
$j_1\leq\ldots\leq j_t$ and $i_c\geq i_{c+1}$ whenever $j_c=j_{c+1}$, proving~\eqref{standard}.

Now,
if $j_{k_1+\cdots+k_a}=j_{k_1+\cdots+k_a+1}-1$ for some $1\leq a\leq m-1$, then by~\eqref{bijectionji}
we have $f_{m-a}=f_{m-a+1}+k_a$ and, by (SK3)
\begin{align*}
i_{k_1+\cdots+k_a}
=f_{m-a+1}-\ga^{(a)}_{k_a}+k_a
=f_{m-a}-\ga^{(a)}_{k_a}
\geq f_{m-a}-\ga^{(a+1)}_1+1
=i_{k_1+\cdots+k_a+1}.
\end{align*}
This and \eqref{bijectionji} show that $\udu_{\mathfrak s}=\udu_f$ and $\udga_{\mathfrak s}=\udga$
if we put $\fks=\flat{{j_1j_2\ldots j_n}\choose{i_1i_2\ldots i_n}}$. Hence, ${\rm im}(\eta_{f})\supseteq\mc K^s_{f^*}(r)$.
\end{proof}

\subsection{Irreducible objects in $\afHr$-${\sf mod}^{\Z}$.} For any given multisegment $\fks\in\mathscr S^{\Z}_{r}$, let $\und\ga_\fks$ and $\udu_\fks$ be defined as in \eqref{ga_s} and \eqref{us}.
\begin{thm} \label{thm:VsDla} If $\fks\in\mathscr S^{\Z}_{r}$, then the irreducible $\afHr$-module $V_\fks$
is isomorphic to  the inflated $\afHr$-module $D^{\und\ga_\fks'}$. In particular,
 $V_\fks$ affords an $\mc H_{\underline{u}_\fks}(r)$-module.\end{thm}
\begin{proof}
Let $\bfs\in\fks$ be standard. By Proposition \ref{prop:stoK}(1), $\und\ga_\fks\in\mc K_{\udu_{\mfs}^{-1}}(r)$ and $\bfs=\bfs^{\sfc}_{\und\ga'_\fks;\udu_\fks}$.
Thus, by Corollary~\ref{cor:hdSpecht}, the Specht module $S^{\und\ga'_\fks}$ has the irreducible head $D^{\und\ga'_\fks}$. On the other hand, by
Lemma~\ref{lem:Ibfs}, there exists a surjective $\afHr$-homomorphism
$I_{\bfs}\twoheadrightarrow S^{\und\ga'_\fks}$, which extends to a surjective
$\afHr$-homomorphism $
I_{\bfs}\twoheadrightarrow D^{\und\ga'_\fks}$.
This together with Proposition~\ref{prop:Ro3} gives the required isomorphism.
The last assertion is clear since $D^{\und\ga'_\fks}$ is an $\mc H_{\underline{u}_\fks}(r)$-module.\end{proof}

In the case $m=1, u_1=1$, Theorem~\ref{thm:VsDla} recovers \cite[Proposition 8.2]{DF}.


We now characterise the Kleshchev multipartition $\udga_\fks$ constructed from a standard word $w$ with $\fks=\flat(w)$. For $r\geq 1$, set
$$
\mc K^s(r):=\dot{\bigcup}_{f\in\mathfrak F}\mc K^s_{f}(r)=\dot{\bigcup}_{f\in\mathfrak F}\mc K^s_{f^*}(r),
$$
where $\dot{\cup}$ means a disjoint union.
Observe that
$
\mathscr S^\Z_r=\dot{\bigcup}_{f\in\mathfrak F}
\mathscr S^\Z_{r,f}.
$

For $\udga\in \mc K^s_{\udu}(r)$, recall that $D^{\udga'}$ is the irreducible
$\mc H_{\udu^{-1}}(r)$-module associated to $\udga'$ defined in Proposition~\ref{prop:BK}.

\begin{thm}\label{thm:standardK}
The bijective maps $\eta_{f}$ ($f\in\mathfrak F$) give rise to a bijection
$$\aligned
\eta:\mathscr S^\Z_r&\longrightarrow \mc K^s(r),\quad
\mathfrak s\longmapsto \udga_{\mathfrak s}.\endaligned$$
Moreover, the set $\{D^{\udga'}~|~\udga\in\mc K^s(r)\}$ after inflation forms
 a complete set of non-isomorphic irreducible objects in the category $\afHr$-{\sf mod}$^\Z$.
\end{thm}
\begin{proof}
By Proposition~\ref{prop:stoK}(2), the first assertion is clear.
For $\udga\in\mc K^s_{f}(r)$, suppose $\eta^{-1}(\udga)=\mathfrak s$.
Then $\udga=\udga_{\mfs},\udu_f= \udu_{\mathfrak s}^{-1}$ and by Theorem~\ref{thm:VsDla} we have
$
D^{\udga'_{\mathfrak s}}\cong V_{\mathfrak s}.
$
Hence,
the set $\{D^{\udga'}~|~\udga\in\mc K^s(r)\}$
coincides with the set of modules $V_{\mathfrak s}$ with $\mathfrak s\in\mathscr S^\Z_r$ up to isomorphism.
Then the last assertion follows from Proposition~\ref{prop:Ro2}(3) by restricting to $\mathscr S_r^\Z$.
\end{proof}

\section{Identification of irreducible representations
of integral type}
In this section, we shall consider the inflation of irreducible representations
of Ariki-Koike algeras of integral type via the canonical surjective homomorphism
$\varpi_{\udu}:\afHr\to\AKa$ given in \eqref{pi}
and identify the corresponding integral multisegments. We will extend the isomorphisms
$D^{\udla'}\cong V_\fks$ for $\udla\in\mc K_f^s(r)$, where $\fks=\overline{\bfs_{\udla';f}^{\sfc}}$
(see Proposition~\ref{prop:stoK}(1) and Theorem~\ref{thm:VsDla}), to all $\udla\in\mc K_f(r)$; see Definition \ref{KMP}.

\subsection{Row residual segments of a multipartition {\rm (cf.\,\S4.1)}.}
Given a sequence of integral segments  $\bfs=(\sfs_1,\sfs_2,\ldots,\sfs_t)$ ($t\geq 1$)
with $\sfs_k=(q^{2i_k},\ldots,q^{2j_k})$, by Definition \ref{defn:order}, $\bfs$ is anti-standard  if it satisfies
\begin{equation}\label{orderVa}
i_1\geq i_2\geq\cdots\geq i_t,\text{ and } j_k\leq j_{k+1}
\text{ whenever } i_k=i_{k+1}.
\end{equation}
This order is exactly the so-called {\em right order} introduced in \cite[\S 2.3]{Va}.
Observe that any multisegment  $\mfs\in\mathscr S^\Z_r$ contains
a unique anti-standard sequence $\bfs=(\sfs_1,\ldots,\sfs_t)\in\mfs$.

For a multisegment $\mfs=\bar{\bfs}\in\mathscr S_r$, set $\mfs^{-1}=\overline{\bfs^{-1}}=\overline{(\sfs_1^{-1},\ldots,\sfs_t^{-1})}$.
By Definition \ref{defn:order}(3), the map
$$(\;)^{-1}:\mathscr S_r^\Z\longrightarrow\mathscr S_r^\Z$$
is a bijection.

For $\mfs\in\mathscr S^\Z_r$, assume that $\bfs=(\sfs_1,\ldots,\sfs_t)\in\mfs$ is anti-standard
and define
\begin{equation}\label{Ms}
M_\mfs:={\rm hd}({\rm ind}^{\afHr}_{\mc H_{\vtg}(\mu)}\C^+_{\sfs_1}\otimes\cdots\otimes\C^+_{\sfs_t})={\rm hd}({\rm ind}^{\afHr}_{\mc H_{\vtg}(\mu)}I^+_\bfs)
\end{equation}
(see \eqref{trivial-sign} and compare \eqref{Is}), where $\mu=(|\sfs_{1}|,\ldots,|\sfs_{t}|)$.

Given $\udla=(\la^{(1)},\ldots,\la^{(m)})\in\mc P_m(r)$ and $f\in\mathfrak F_m$, define
the {\it i-th row residual segment} of $\la^{(k)}$ relative to $\udu=\udu_f$:
$$
\sfs^{(k)}_i=(q^{2(f_k+1-i)},q^{2(f_k+2-i)},\ldots,q^{2(f_k+\la^{(k)}_i-i)})\quad(1\leq k\leq m,1\leq i\leq\ell(\la^{(k)}))$$
by reading the residues of the nodes in
the $i$th row of $\la^{(k)}$ from left to right, and then form the sequence of segments in the order of row 1, row 2, ... of the first partition of $\udla$, of the second partition, and so on:
\begin{equation}\label{srla}
\bfs^{\sfr}_{\udla;f}=(\sfs^{(1)}_1,\ldots,\sfs^{(1)}_{\ell(\la^{(1)})},
\ldots,\sfs^{(m)}_1,\ldots,\sfs^{(m)}_{\ell(\la^{(m)})}).
\end{equation}

By~\eqref{scla} and~\eqref{srla}, a direct calculation shows that
\begin{equation}\label{c&r}
(\bfs^{\sfc}_{\udla;f})^{-1}=\bfs^{\sfr}_{\udla';f^*}
\end{equation}
for any $\udla\in\mc P_m(r)$ and $f\in\mathfrak F_m$.

\begin{prop}[{\cite[Theorem 3.4]{Va}}]\label{prop:Va}
Let $\udla=(\la^{(1)},\ldots,\la^{(m)})\in\mc K_{f}(r)$ with $f\in\mathfrak F_m$ and let $\la=(|\la^{(1)}|,\ldots,|\la^{(m)}|)$.
Assume that $\mfs$ is the multisegment containing $\bfs^{\sfr}_{\udla;f}$.
Then the $\afHr$-module ${\rm hd}\big({\rm ind}^{\afHr}_{\mc H_{\vtg}(\la)}
({\rm ev}^*_{u_1}(S^{\la^{(1)}})\otimes\cdots\otimes {\rm ev}^*_{u_m}(S^{\la^{(m)}}))\big)$
is irreducible and isomorphic to $M_{\fks}$.
\end{prop}

\begin{rem}
It is straightforward to check that there exists an automorphism $\sharp$ on $\afHr$ defined by
\begin{align}
\sharp: \afHr&\longrightarrow\afHr,\quad
T_i\longmapsto -q^2T_i^{-1},\quad
X_k\longmapsto X_k^{-1}\label{sharp}
\end{align}
for $1\leq i\leq r-1$ and $1\leq k\leq r$.
Given an $\afHr$-module $M$, we denote by $M^{\sharp}$ the $\afHr$-module obtained from $M$ by
twisting the action via the automorphism $\sharp$.
It is easy to deduce that for
an $\mc H_{\vtg}(a)$-module $M$ and an $\mc H_{\vtg}(b)$-module $N$ with
$1\leq a,b\leq r$ and $a+b=r$,
we have
\begin{equation}\label{indsharp}
\big({\rm ind}^{\afHr}_{\mc H_{\vtg}(a)\otimes\mc H_{\vtg}(b)}(M\otimes N)\big)^\sharp
\cong {\rm ind}^{\afHr}_{\mc H_{\vtg}(a)\otimes\mc H_{\vtg}(b)}(M^\sharp\otimes N^\sharp).
\end{equation}
\end{rem}

\subsection{Identification of irreducible representations.} We now prove the main result of this section.
\begin{lem}\label{lem:VsMs}
For any $\mfs\in\mathscr S^\Z_r$, we have
$
M^{\sharp}_{\mfs^{-1}}\cong V_\fks.
$
\end{lem}
\begin{proof} 
Let $\bfs=(\sfs_1,\ldots,\sfs_t)\in\mfs$ be standard.
By Definition~\ref{defn:order},
$\bfs^{-1}=(\sfs_1^{-1},\ldots,\sfs_t^{-1})$ is the unique anti-standard element in $\mfs^{-1}$
and hence, by~\eqref{Ms},
$
M_{\mfs^{-1}}=
{\rm hd}({\rm ind}^{\afHr}_{\mc H_{\vtg}(\mu)}(I_{\bfs^{-1}}^+), 
$ where $\mu=(|\sfs_1^{-1}|,\ldots,|\sfs_t^{-1}|)$.
This together with~\eqref{indsharp} implies
$
M^\sharp_{\mfs^{-1}}\cong
{\rm hd}({\rm ind}^{\afHr}_{\mc H_{\vtg}(\mu)}(\C^+_{\sfs_1^{-1}})^\sharp\otimes\cdots\otimes(\C^+_{\sfs_t^{-1}})^\sharp).
$
Observe that $(\C^+_{\sfs^{-1}})^\sharp\cong \C_{\sfs}$ for any segment $\sfs$
and hence,
\begin{align*}
M^\sharp_{\mfs^{-1}}
\cong
{\rm hd}({\rm ind}^{\afHr}_{\mc H_{\vtg}(\mu)}\C_{\sfs_1}\otimes\cdots\otimes\C_{\sfs_t})
={\rm hd} (I_{\bfs})\cong V_\bfs\cong V_{\mfs},
\end{align*}
where the second isomorphism follows from~Proposition \ref{prop:Ro3}
since $\bfs$ is standard.
\end{proof}


Observe also that the automorphism $\sharp$ on $\afHr$ also defines an automorphism
on the Hecke algebra $\Ha$, denoted again by $\sharp$, which maps $T_w$ to $(-q^2)^{l(w)}(T_{w^{-1}})^{-1}$
for $w\in\mf S_r$. Since $q$ is not a root of unity, each Specht module $S^\la$
is self-dual for $\la\in\mc P(r)$ with respect to the automorphism on $\Ha$ induced by $*:T_w\mapsto T_{w^{-1}}$ for $w\in\mf S_r$.
Furthermore it is known \cite[Theorem 3.5]{DJ2} that
\begin{equation}\label{Spechtsharp}
(S^{\la})^\sharp\cong S^{\la'}
\end{equation}
for $\la\in\mc P(r)$.

Recall that for $\udla\in\mc K_{f^*}(r)$ with $f\in\mathfrak F$,
$D^{\udla'}$ is an irreducible $\mc H_{\udu_f}(r)$-module, which can be also regarded as an
$\afHr$-module by inflation.

We are now ready to generalise Theorem \ref{thm:VsDla} to arbitrary Kleshchev multipartitions.
\begin{thm}\label{thm:identify}
Suppose $\udla\in\mc P_m(r)$ such that $\udla'\in\mc K_{f^*}(r)$ for some $f\in\mathfrak F_m$.
Then there is an isomorphism of $\afHr$-modules:
$
D^{\und{\la}}\cong V_{\bfs^{\sfc}_{\und{\la};f}}.
$

\end{thm}
\begin{proof}
By Lemma~\ref{lem:VsMs} and~\eqref{c&r}, the following holds
\begin{equation}\label{Vscla}
V_{\bfs^\sfc_{\und{\la};f}}
\cong M^\sharp_{\bfs^\sfr_{\udla';f^*}}.
\end{equation}
Since $\udla'\in\mc K_{f^*}(r)=\mc K_{\udu^{-1}_f}(r)$, by Proposition~\ref{prop:Va} we obtain
$$
M_{\bfs^{\sfr}_{\udla';f^*}}
\cong {\rm hd}\big({\rm ind}^{\afHr}_{\mc H_{\vtg}(\mu)}
{\rm ev}^*_{u_m^{-1}}(S^{\la^{(m)'}})\otimes\cdots\otimes
{\rm ev}^*_{u_1^{-1}}(S^{\la^{(1)'}})\big),
$$
where $\mu=(|\la^{(m)}|,\ldots,|\la^{(1)}|)$ and $u_k=q^{2f_k}$ for $1\leq k\leq m$.
This together with~\eqref{Vscla} and~\eqref{indsharp} gives rise to
\begin{align}
V_{\bfs^\sfc_{\und{\la};f}}
&\cong {\rm hd}({\rm ind}^{\afHr}_{\mc H_{\vtg}(\mu)}
{\rm ev}^*_{u_m^{-1}}(S^{\la^{(m)'}})\otimes\cdots\otimes
{\rm ev}^*_{u_1^{-1}}(S^{\la^{(1)'}}))^\sharp\notag\\
&\cong {\rm hd}~{\rm ind}^{\afHr}_{\mc H_{\vtg}(\mu)}
{\rm ev}^*_{u_m^{-1}}(S^{\la^{(m)'}})^\sharp\otimes\cdots\otimes
{\rm ev}^*_{u_1^{-1}}(S^{\la^{(1)'}})^\sharp. \label{Vsclaind}
\end{align}
By~\eqref{ev} and~\eqref{sharp}, one can easily check that
$
{\rm ev}^*_u(M)^\sharp \cong {\rm ev}^*_{u^{-1}}(M^\sharp)
$
for any $\Ha$-module $M$ and hence, by~\eqref{Spechtsharp},
$
{\rm ev}^*_u(S^\la)^\sharp \cong {\rm ev}^*_{u^{-1}}((S^\la)^\sharp)
\cong {\rm ev}^*_{u^{-1}}(S^{\la'}).
$
This means
\begin{align*}
{\rm ind}^{\afHr}_{\mc H_{\vtg}(\mu)}
({\rm ev}^*_{u_m^{-1}}(S^{\la^{(m)'}})^\sharp\otimes\cdots\otimes
{\rm ev}^*_{u_1^{-1}}(S^{\la^{(1)'}})^\sharp)
\cong {\rm ind}^{\afHr}_{\mc H_{\vtg}(\mu)}
({\rm ev}^*_{u_m}(S^{\la^{(m)}})\otimes\cdots\otimes
{\rm ev}^*_{u_1}(S^{\la^{(1)}})).
\end{align*}
Since $\udla'\in\mc K_{\udu_f^{-1}}(r)=\mc K_{f^*}(r)$, it follows from \eqref{Vsclaind} and Proposition \ref{prop:BK}(1) that
\begin{align*}
V_{\bfs^{\sfc}_{\und{\la};f}}
\cong {\rm hd}~{\rm ind}^{\afHr}_{\mc H_{\vtg}(\mu)}
({\rm ev}^*_{u_m}(S^{\la^{(m)}})\otimes\cdots\otimes
{\rm ev}^*_{u_1}(S^{\la^{(1)}}))
=D^{\udla},
\end{align*}
proving the theorem.
\end{proof}
\subsection{A branching property.} When the parameter $q$ is not a root of unity, the Hecke algebra $\Ha$ is semisimple. Thus, in this case, the restriction to $\Ha$ of an irreducible $\afHr$-module $V_\fks$ is a direct sum of Specht modules $S^\la$, $\la\vdash r$. The determination of $S^\la$ which appears in $V_\fks$ is called the affine branching rule in \cite[Problem 4.3.6]{DDF}.

We now look at a relatively simple problem. For $\und{\la}=(\la^{(1)},\ldots,\la^{(m)})\in\mc P_m(r)$,
define $\udla^+=(\udla^+_1,\udla^+_2,\ldots)\in\mc P(r)$ by setting
$
\udla^+_i=\la^{(1)}_i+\la^{(2)}_i+\cdots+\la^{(m)}_i, \text { for }i\geq 1,
$
i.e.,
$$\udla^+=\la^{(1)}+\la^{(2)}+\cdots+\la^{(m)}.$$

Generally speaking, by~\eqref{slaind} and the Littlewood-Richardson rule, one can deduce that
the Specht module $S^{\udla^+}$ over $\Ha$
occurs in the restriction $S^{\und{\la}}|_{\Ha}$ with multiplicity one. (One can also prove this by an induced cell module argument; see \cite[Theorem 2.8]{Roi}.) If, in addition, $\udla'$ is a Kleshchev multipartition,
does $S^{\udla^+}$ appear in $D^{\udla}|_{\Ha}$? This sounds like a simple question but we didn't see a straightforward proof in literature.

\begin{cor}\label{cor:quotient}
Let $\und{\la}\in\mc P_m(r)$
such that $\udla'\in\mc K_{f^*}(r)$ with $f\in\mathfrak F_m$.
Then
\begin{enumerate}
\item $V_{\bfs^\sfc_{\und{\la};f}}$ is a quotient of $I_{\bfs^\sfc_{\und{\la};f}}$.

\item The Specht $\Ha$-module $S^{\udla^+}$  occurs
with multiplicity one in the restriction $D^{\und{\la}}|_{\Ha}$.
\end{enumerate}
\end{cor}
\begin{proof}
Part (1)  follows from Lemma~\ref{lem:Ibfs} and Theorem~\ref{thm:identify}. 
By Proposition~\ref{prop:Ro2}(3) and Theorem~\ref{thm:identify}, 
 the cell module $E_{\mu(\bfs^\sfc_{\und{\la};f})}$
occurs with multiplicity one in $D^{\und{\la}}$ as $\Ha$-modules.
Observe that, by~\cite[I, (1.8)]{Mac} and~\eqref{mubfs},
$
\udla^+=((\und{\la}')^\vee)'=\mu(\bfs^\sfc_{\und{\la};f})'
$
and, hence, $S^{\udla^+}\cong E_{\mu(\bfs^c_{\und{\la};f})}$ occurs with multiplicity one in $D^{\und{\la}}|_{\Ha}$.
\end{proof}
\begin{rem}\label{rem:Vazirani}
It would be interesting to find a direct approach to verifying that the Specht module
$S^{\udla^+}$ must occur in $D^{\und{\la}}|_{\Ha}$
whenever $D^{\und{\la}}$ is nonzero.
Then combining this with Lemma~\ref{lem:Ibfs} gives an elementary
way to determine the multisegments corresponding to the irreducible modules associated
to Kleshchev multipartitions without using Proposition~\ref{prop:Va} of  Vazirani.



\end{rem}
\section{Drinfeld polynomials of integral type}
In this section, we shall compute the Drinfeld polynomials
associated with the irreducible $\rmU_q(\afgl)$-modules which are inflated from the irreducible $\AKa$-modules associated with Kleshchev multipartitions.

\subsection{\bf The affine Schur-Weyl duality. }
Let $\rmU_q(\afgl)$ be the Drinfeld's {\it quantum loop algebra}, which is the $\C$-algebra generated by $\ttx^\pm_{i,s}$
($1\leq i<n$, $s\in\Z$), $\ttk_i^{\pm1}$ and $\ttg_{i,t}$ ($1\leq
i\leq n$, $t\in\Z\backslash\{0\}$) subject to certain explicit relations (which we do not need here); see, e.q., \cite[\S2.5]{DDF}.
Let $V_n$ be an $n$-dimensional vector space over the $\C$-space
and set
$$\Og_{(n)}=V_n\otimes_{\C}\C[X,X^{-1}]=V_n[X,X^{-1}].$$
Following \cite{VV}, there is a right action of $\afHr_\C$ on $\Og_{(n)}^{\ot r}$.
Then the endomorphism algebra
%
%
%
\begin{align}
\afSr:={\rm End}_{\afHr}(\Og_{(n)}^{\ot r})
\end{align}
 is known as the { affine $q$-Schur algebra}.  In particular,
$\Og_{(n)}^{\ot r}$ becomes an $\afSr$-$\afHr$-bimodule.
By a double Hall algebra interpretation of $\rmU_q(\afgl)$
(see \cite[\S 3.5]{DDF}), there is a left action
of $\rmU_q(\afgl)$-module on $\Og_{(n)}^{\ot r}$ which
commutes with the right action of $\afHr$ on $\Og_{(n)}^{\ot r}$.
Furthermore, this gives rise to an algebra homomorphism
\begin{align}\label{afglact}
\zeta_r:\rmU_q(\afgl)\longrightarrow {\rm End}_{\afHr}(\Og_{(n)}^{\ot r})=\afSr,
\end{align}
which is surjective, see \cite[Corollary 3.8.4, Remark 3.8.5(2)]{DDF}.
We omit the details of the actions of $\afHr$ and $\rmU_q(\afgl)$
on $\Og_{(n)}^{\ot r}$ and the construction for $\zeta_r$ here; see \cite[\S3.4,3.5]{DDF}.

\subsection{\bf Identification of irreducible representations of $\mc S_\vtg(n,r)$. }
For $1\leq i\leq n$ and $s\in\Z$, define the elements $\ms
Q_{i,s}\in\rmU_q(\afgl)$ through the generating functions
\begin{equation}\label{exp}
\begin{split}
&\quad\qquad\ms Q_i^\pm(x):=\exp\bigg(-\sum_{t\geq
1}\frac{1}{[t]_q}\ttg_{i,\pm t} (qx)^{\pm t}\bigg)=\sum_{s\geq
0}\ms Q_{i,\pm s} x^{\pm s}\in\rmU_q(\afgl)[[x,x^{-1}]],
\end{split}
\end{equation}
where $\ds [t]_q=\frac{q^n-q^{-n}}{q-q^{-1}}$.
For a finite dimensional $\rmU_q(\afgl)$-module $V$ and $\la=(\la_1,\ldots,\la_n)\in\Z^n$,
a nonzero ($\la$-weight) vector $w\in V$ is called a {\it pseudo-highest weight
vector}\index{pseudo-highest weight vector} if there exist
some $Q_{i,s}\in\C$ such that
\begin{equation}\label{HWvector}
\ttx_{j,s}^+w=0,\quad\ms Q_{i,s}w=Q_{i,s}w,\quad\text{and}\quad
\ttk_iw=q^{\la_i}w,
\end{equation}
for all $1\leq i\leq n$, $1\leq j< n$, and $s\in\Z$. A $\rmU_q(\afgl)$-module
 $V$ is called a {\it pseudo-highest weight module} if $V=\rmU_q(\afgl) w$
for some pseudo-highest weight vector $w$.

Following \cite{FM}, an $n$-tuple of polynomials
$\bfQ=(Q_1(x),\ldots,Q_n(x))$ with constant terms $1$ is called {\it
dominant} if, for each $1\leq i\leq n-1$, the ratio
$Q_i(xq^{i-1})/Q_{i+1}(xq^{i+1})$ is a polynomial in $x$. Let
$\sQ(n)$ be the set of dominant $n$-tuples of polynomials.

For a polynomial $Q(x)=\prod_{1\leq i\leq m}(1-a_ix)\in\C[x]$
with $a_i\in\C^*$, put $Q^+(x)=Q(x)$ and define
$$Q^-(x)=\prod_{1\leq i\leq m}(1-a_i^{-1}x^{-1})\in\C[x^{-1}].$$
 Given a $\bfQ=(Q_1(x),\ldots,Q_{n}(x))\in\sQ(n)$, define
$Q_{i,s}\in\C$, for $1\leq i\leq n$ and $s\in\Z$, by the
following formula
\begin{equation}\label{Qis}
Q_i^\pm(x)=\sum_{s\geq 0}Q_{i,\pm s}x^{\pm s}.
\end{equation}
Let $I(\bfQ)$ be the left ideal of $\rmU_q(\afgl)$ generated by $\ttx_{j,s}^+ ,\ms
Q_{i,s}-Q_{i,s},$ and $\ttk_i-q^{\la_i}$, for $1\leq j<n$,
$1\leq i\leq n$, and $s\in\Z$, where $\la_i=\mathrm{deg\,}Q_i(x)$,
and define
\begin{equation}\label{MQ}
M(\bfQ)=\rmU_q(\afgl)/I(\bfQ).
\end{equation}
Then $M(\bfQ)$ has a unique irreducible quotient, denoted by $L(\bfQ)$.
The polynomial $Q_i(x)$ are called {\em Drinfeld polynomials} associated to $L(\bfQ)$.
Clearly, $L(\bfQ)$ is a pseudo-highest weight module with pseudo-highest weight
$\la=(\la_1,\ldots,\la_n)\in\La(n,r)$, where $\La(n,r)$ denotes the set
of compositions of $r$ into $n$ parts.

Let
$$\sQ(n)_r=\big\{\bfQ=(Q_1(x),\ldots,Q_n(x))\in\sQ(n)\mid r=\sum_{1\leq i\leq n}\mathrm{deg}\, Q_i(x)\big\}.$$

Observe that given a left $\afHr$-module $M$, the tensor product
$\Og_{(n)}^{\ot r}\otimes_{\afHr} M$ naturally affords a left $\afSr$-module.
Recall that $\{V_{\mfs}\mid\mfs\in\mathscr S_r\}$ is a complete set of irreducible $\afHr$-modules.
Further, set
$$
\mathscr S_r^{(n)}=\{\fks=
\{\sfs_1,\ldots,\sfs_t\}\in\mathscr S_r\mid t\geq 1,\,|\sfs_i|\leq n,\,\forall i\}.
$$
Clearly, if $n\geq r$, then $\mathscr S_r^{(n)}=\mathscr S_r$.

Given $n,r\geq1$ and $\mathfrak s\in\mathscr S_r^{(n)}$,
take $\bfs=(\sfs_1,\sfs_2,\ldots,\sfs_t)\in\mathfrak s$
with
\begin{equation}\label{sfs_k}
\sfs_k=(z_kq^{2i_k},z_kq^{2(i_k+1)},\ldots,z_kq^{2j_k})\in(\C^*)^{r_k},
\end{equation}
where $r_k=|\sfs_k|=j_k-i_k+1\in\Z_+$ for $1\leq k\leq t$.
We define
\begin{equation}\label{Qs}
\bfQ_{\mathfrak s}=(Q_1(x),\ldots,Q_n(x))\in \sQ(n)_r
\end{equation}
 by setting recursively
\begin{equation}\label{partialQ}
Q_k(x)
=\prod_{k\leq l\leq n}P_{l}(xq^{l+1-2k})
=P_k(xq^{-k+1})P_{k+1}(xq^{-k+2})\cdots
P_n(xq^{n+1-2k}),
\end{equation}
where
\begin{equation}\label{partialP}
P_{k}(x)=\prod_{1\leq b\leq t\atop r_b=k}(1-z_bq^{i_b+j_b}x)
\end{equation}
for $1\leq k\leq n$.
Here we use the convention that $P_i(x)=1$ if there does not exist $1\leq j\leq t$
such that $r_j=i$.
This gives a map
$$\partial: \mathscr S^{(n)}_r\longrightarrow \sQ(n)_r,\quad \mathfrak s\longmapsto \partial(\mathfrak s)=\bfQ_{\mathfrak s}.$$
Thus, we obtain the associated irreducible $\rmU_q(\afgl)$-module $L(\bfQ_{\mathfrak s})$.
The following result is due to Deng, Du and Fu \cite{DD,DDF,F}.

\begin{prop}[{\cite[Theorem 6.6]{DD}, \cite[Theorem 4.4]{DF}}] \label{prop:DF&DD}
Assume $n,r\geq 1$.
\begin{enumerate}
\item
Both the set $\big\{L(\bfQ)\mid\bfQ\in\sQ(n)_r\big\}$
and the set
$
\{\Og_{(n)}^{\ot r}\otimes_{\afHr} V_{\mathfrak s}~|\mathfrak s\in\mathscr S_r^{(n)}\}
$
are complete sets of non-isomorphic irreducible representations of $\afSr$.
\item
For each $\mathfrak s\in\ms S_r^{(n)}$, there is a
$\rmU_q(\afgl)$-module isomorphism (or equivalently, an $\afSr$-module isomorphism)
$$
L(\bfQ_{\mathfrak s})\cong \Og_{(n)}^{\ot r}\ot_{\afHr} V_{\mathfrak s}.
$$
Hence, the map $\partial$ is a bijection.
\end{enumerate}
\end{prop}

\subsection{Drinfeld polynomials of integral type.} Recall from \S2.3 that the category $\afHr$-${\sf mod}^{\Z}$
has irreducible objects which are indexed by the set $\mathscr S_r^{\Z}$ of integral multisegments and they completely determine via \eqref{gen&int} the irreducible objects in $\afHr$-${\sf mod}$.

A polynomial $Q(x)\in\C[x]$ of constant term 1 is said to be of {\em integral type} if its roots are powers of $q^2$.
Let $\sQ(n)^{\Z}$ be the set of dominant $n$-tuples of polynomials of integral type.
Set
\begin{align*}
\sQ(n)^{\Z}_r=\sQ(n)^{\Z}\cap \sQ(n)_r,\qquad
\mathscr S_r^{(n),\Z}=\mathscr S^\Z_r\cap\mathscr S^{(n)}_r.
\end{align*}

Given $\mathfrak s\in\mathscr S_r^{(n)}$ with
$\bfs=(\sfs_1,\sfs_2,\ldots,\sfs_t)\in\mathfrak s$, where
$\sfs_k$
for $1\leq k\leq t$ are given in \eqref{sfs_k}.
Let $\bfQ_{\mfs}=(Q_1(x),\ldots,Q_n(x))$ be defined as in \eqref{Qs}.
By~\eqref{partialQ} and~\eqref{partialP}, we have
\begin{align}
Q_k(x)=\prod_{k\leq l\leq n}P_{l}(xq^{l+1-2k})
=&\prod_{k\leq l\leq n}\prod_{1\leq b\leq t,\atop j_b-i_b+1=l}(1-z_bq^{i_b+j_b+l+1-2k}x)\notag\\
=&\prod_{1\leq b\leq t,\atop j_b-i_b+1\geq k}(1-z_bq^{2(j_b+1-k)}x).
\label{partialQ1}
\end{align}
This implies that $\bfQ_{\mathfrak s}\in\sQ(n)^{\Z}_r$
if and only if $\mathfrak s\in\mathscr S_r^{(n),\Z}$.
Hence, the restriction of the map $\partial$ gives a bijection
\begin{equation}\label{partialint}
\partial^\Z: \mathscr S_r^{(n),\Z}\longrightarrow\sQ(n)^{\Z}_r.
\end{equation}

For $a\in\C^*$, analogous to~\eqref{sigma}, consider the Hopf algebra automorphism
\begin{equation}\omega_a:\rmU_q(\afgl)\longrightarrow \rmU_q(\afgl),\quad
\ttx^\pm_{i,s}\mapsto a^s\ttx^\pm_{i,s},\quad \ttg_{j,t}\mapsto a^t\ttg_{j,t}, \quad\ttk_j^{\pm1}\mapsto \ttk_j^{\pm1}\label{omega}.
\end{equation}
Thus, we may twist a $\rmU_q(\afgl)$-module $M$ by $\omega_a$ to get a
new module denoted by $M^{\omega_a}$.
For $\bfQ=(Q_1(x),\ldots,Q_n(x))\in\sQ(n)_r$, define
$\bfQ^{\omega_a}=(Q^{\omega_a}_1(x),\ldots,Q^{\omega_a}_n(x))$ by setting
\begin{equation}\label{Qomega}
Q^{\omega_a}_k(x)=Q_k(ax)
\end{equation}
for $1\leq k\leq n$.
Analogous to~\eqref{Vs-sigma}, we have the following.
\begin{lem}\label{LQomega}
Suppose $a\in\C^*$ and $\bfQ=(Q_1(x),\ldots,Q_n(x))\in\sQ(n)_r$.
Then
$$
L(\bfQ)^{\omega_a}\cong L(\bfQ^{\omega_a}).
$$
\end{lem}
\begin{proof}
Let $w$ be a pseudo-highest weight vector of $L(\bfQ)$.
By~\eqref{omega}, \eqref{exp} and~\eqref{HWvector},  $\omega_a(\ms Q_{i,s})=a^s\ms Q_{i,s}$
and hence
$$
\omega_a(\ttx^+_{j,s})w=0,\quad \omega_a(\ms Q_{i,s})w=a^s\ms Q_{i,s}w=a^sQ_{i,s}w, \quad \omega_a(\ttk_i)w=q^{\la_i}w,
$$
where $\la_i={\rm deg}\,Q_i(x)$ and $Q_{i,s}$ is defined in \eqref{Qis} for
$1\leq i\leq n, 1\leq j<n$ and $s\in\Z$.
This implies that $w$ is also a pseudo-highest weight vector of $L(\bfQ)^{\omega_a}$ and moreover
$L(\bfQ)^{\omega_a}$ is a quotient of $M(\bfQ^{\omega_a})$ by~\eqref{MQ} and~\eqref{Qomega}.
Then the lemma is proved since $L(\bfQ)^{\omega_a}$ is irreducible.
\end{proof}


 It is known from \cite[\S4.3]{FM} that $L(\bfQ)$ for $\bfQ\in\sQ(n)$ are all non-isomorphic
finite dimensional irreducible polynomial representations of $\rmU_q(\afgl)$. (We refer the reader to \cite{DDF,FM}
for the definition of polynomial representations of $\rmU_q(\afgl)$.)
We now use the discussion in \S2.3 and Proposition 6.2.1 to prove the following (cf. \cite[Section 3.7]{HL}).

\begin{theorem} \label{thm:intDrinfeld}
For $\bfQ\in\sQ(n)$, there exist $\bfQ^{(1)},\ldots,\bfQ^{(p)}\in\sQ(n)^{\Z}$ and $a_1,\ldots,a_p\in\C^*$
such that as $\rmU_q(\afgl)$-modules
$$
L(\bfQ)\cong L(\bfQ^{(1)})^{\omega_{a_1}}\otimes\cdots\otimes L(\bfQ^{(p)})^{\omega_{a_p}}.
$$
\end{theorem}
\begin{proof}
Write $\bfQ=(Q_1(u),\ldots,Q_n(u))$ and let $r=\sum_i{\rm deg}Q_i(u)$.
By Proposition~\ref{prop:DF&DD},
there exists $\mfs\in\mathscr S^{(n)}_r$ such that $\bfQ=\bfQ_{\mfs}$ and
\begin{equation}\label{LQ&Vs}
L(\bfQ)\cong \Og_{(n)}^{\ot r}\ot_{\afHr} V_{\mfs}.
\end{equation}
Suppose $\bfs\in\mfs$ is
standard in the sense of Definition~\ref{defn:order}(2)
with the chain decomposition
$\bfs=\bfs^{(1)}\vee\bfs^{(2)}\vee\cdots\vee\bfs^{(p)}$, where every
$\bfs^{(k)}$ is a chain lying on the line $L_{a_k}$ for $1\leq k\leq p$.
Let $\mfs^{(k)}$ be the multisegment containing $\bfs^{(k)}$ for $1\leq k\leq p$.
By~\eqref{gen&int}, \eqref{LQ&Vs}, and Proposition~\ref{prop:DF&DD}, we have the following isomorphisms of $\rmU_q(\afgl)$-modules
\begin{align}
L(\bfQ)&\cong \Og_{(n)}^{\ot r}\ot_{\afHr}
{\rm ind}^{\afHr}_{\mc H_{\vtg}(\mu)}(V_{\mfs^{(1)}}\otimes\cdots\otimes V_{\mfs^{(p)}})\notag\\
&\cong\Og_{(n)}^{\ot r}\ot_{\mc H_{\vtg}(\mu)}\big(V_{\mfs^{(1)}}\otimes\cdots\otimes V_{\mfs^{(p)}}\big)\notag\\
&\cong \big(\Og_{(n)}^{\ot r_1}\otimes_{\mc H_{\vtg}(r_1)}V_{\mfs^{(1)}}\big)
\otimes\cdots\otimes\big(\Og_{(n)}^{\ot r_p}\otimes_{\mc H_{\vtg}(r_p)}V_{\mfs^{(p)}}\big)\notag\\
&\cong L(\bfQ_{\mfs^{(1)}})\otimes\cdots\otimes L(\bfQ_{\mfs^{(p)}}), \label{LQtensor}
\end{align}
where $\mu=(r_1,\ldots,r_p)$ and $r_k=\sum_i|\sfs^{(k)}_i|$ for $1\leq k\leq p$.

Furthermore, set
\begin{equation}\label{intsegment}
\bfs^{(k)'}=(\bfs^{(k)})^{\sigma_{a_k^{-1}}}
\end{equation}
for $1\leq k\leq p$.
Then $\mfs^{(k)'}:=\overline{\bfs^{(k)'}}\in\mathscr S^\Z_{r_k}$ and, hence,
$
\bfQ^{(k)}:=\bfQ_{\mfs^{(k)'}}\in\sQ(n)_{r_k}^{\Z}
$
for $1\leq k\leq p$ are of integral type.
Then, by~\eqref{intsegment}, \eqref{partialQ1} and \eqref{s-sigma}, one can easily deduce that
$$
\bfQ_{\mfs^{(k)}}=\bfQ_{(\mfs^{(k)'})^{\sigma_{a_k}}}=\bfQ^{\omega_{a_k}}_{\mfs^{(k)'}}=(\bfQ^{(k)})^{\omega_{a_k}}
$$
for $1\leq k\leq p$.
This together with Lemma~\ref{LQomega} and~\eqref{LQtensor} gives rise to the desired isomorphism of
$\rmU_q(\afgl)$-modules
$$
L(\bfQ)\cong L(\bfQ^{(1)})^{\omega_{a_1}}\otimes\cdots\otimes L(\bfQ^{(p)})^{\omega_{a_p}}.
$$
\end{proof}

Now consider
the functor
\begin{equation}\label{sF}
{\mc F}:\afHr\text{-{\sf mod}}^{\Z}\longrightarrow\afSr\text{-{\sf mod}},\quad L\longmapsto\Og_{(n)}^{\otimes r}\otimes_{\afHr} L.
\end{equation}
We form full subcategory $\afSr\text{-{\sf mod}}^{\Z}$ whose objects are isomorphic to ${\mc F}(L)$ for $L\in
\text{Ob}(\afHr\text{-{\sf mod}}^{\Z})$. Then we have the following remark as an analog of the result in \S2.3.


\begin{rem}
By Theorem~\ref{thm:intDrinfeld}, to study irreducible $\afSr$-modules
or finite dimensional irreducible polynomial representation of $\rmU_q(\afgl)$, it is enough to consider those $L(\bfQ)$
associated with Drinfeld polynomials $\bfQ$ of integral type.
In other words, irreducible objects in $\afSr\text{-{\sf mod}}^{\Z}$ determine all irreducible objects in $\afSr\text{-{\sf mod}}$ and every irreducible polynomial representation of $\rmU_q(\afgl)$  is determined
by the irreducible objects in the full subcategory $\rmU_q(\afgl)\text{-{\sf mod}}^{\Z}$ which is defined by inflating $\oplus_{r\geq0}\afSr\text{-{\sf mod}}^{\Z}$.
\end{rem}

\subsection{\bf Standard multipartitions and Drinfeld polynomials.}
For $f\in\mathfrak F$, $r\geq 1$,
let
\begin{align*}
\mc K^{s}_{f,(n)}(r)
=\{\udga\in\mc K^s_{f}(r)~|~\ga^{(k)}_1\leq n, \forall k\geq 1\}, \qquad
\mc K^{s}_{(n)}(r)=\dot{\bigcup}_{f\in\mathfrak F}\mc K^{s}_{f,(n)}(r).
\end{align*}
Recall from Theorem \ref{thm:standardK} the bijective map $\theta=\eta^{-1}:\mc K^s(r)\to\mathscr S_r^{\Z}$. Restriction induces a bijection
$$\theta^{\Z}:\mc K^s_{(n)}(r)\to\mathscr S_r^{(n),\Z}.$$

The two parts of the following result are the counterparts of Theorems \ref{thm:standardK} and
\ref {thm:identify}, respectively.

\begin{thm}\label{thm:K&D}
The following holds for $n,r\geq 1$:
\begin{enumerate}
\item There is a bijection
$$
\partial^{\Z}\circ\theta^{\Z}: \mc K^{s}_{(n)}(r)\longrightarrow
\sQ(n)_r^{\Z}
$$
such that if $\bfQ=\partial^{\Z}\circ\theta^{\Z}(\udga)\in\sQ(n)^{\Z}_r$ for some $\udga\in {\mc K}^{s}_{(n)}(r)$, then the following isomorphism of $\rmU_q(\afgl)$-modules (or $\afSr$-modules) holds
$$
L(\bfQ)\cong \Og_{(n)}^{\ot r}\ot_{\afHr} D^{\udga'}.
$$

\item Let $\udla\in\mc P_m(r)$ and $f\in\mathfrak F$.
Assume that $\udla'\in\mc K_{f^*}(r)$ and
$\ell(\la^{(k)})\leq n$ for $1\leq k\leq m$
and write $\bfs=\bfs^{\sfc}_{\udla;f}$ and $\mathfrak s=\bar{\bfs}$.
Then $\bfQ_{\mathfrak s}=(Q_1(x),Q_2(x),\ldots,Q_n(x))$ are given by
\begin{equation}\label{rrr}
Q_i(x)=\prod_{1\leq k\leq m, 1\leq j\leq \la^{(k)}_i}(1-q^{2(f_k+j-i)}x).
\end{equation}
In other words, for every $1\leq i\leq n$,
$Q_i(x)$ is the polynomial whose roots are determined by
the residues (with respect to $\udu_f$)
of the nodes in $i$th row of $\la^{(k)}$ for all $1\leq k\leq m$.
Moreover, the following isomorphism of $\rmU_q(\afgl)$-modules holds:
$$
\Og_{(n)}^{\ot r}\ot_{\afHr} D^{\udla}\cong L(\bfQ_{\mathfrak s}).
$$
\end{enumerate}
\end{thm}
\begin{proof}
Suppose $\bfQ\in\sQ(n)^{\Z}_r$.
By Proposition~\ref{prop:DF&DD} and~\eqref{partialint}, there exists a unique
$\mathfrak s\in\mathscr S^{(n),\Z}_r$ such that $\bfQ=\partial^{\Z}(\mathfrak s)$ and
$L(\bfQ)\cong \Og_{(n)}^{\ot r}\ot_{\afHr} V_{\mathfrak s}$.
Then, by Theorem~\ref{thm:standardK}, we have $\udga=\eta(\mathfrak s)\in\mc K^{s}_{(n)}(r)$ and,
by Theorem~\ref{thm:VsDla}, $V_\fks\cong D^{\udga'}$ and so
$$
L(\bfQ)\cong \Og_{(n)}^{\ot r}\ot_{\afHr}D^{\udga'},
$$
proving (1).

Fix $1\leq i\leq n$.
Recall from~\eqref{scla} that we have
$$
\bfs^{\sfc}_{\und{\la};f}=(\sfs_1^{(m)},\ldots,\sfs^{(m)}_{\la^{(m)}_1},\ldots,
\sfs^{(1)}_1,\ldots,\sfs^{(1)}_{\la^{(1)}_1}),
$$
where
$
\sfs^{(k)}_j=(q^{2(f_k+j-\la^{(k)'}_j)},\ldots, q^{2(f_k+j-2)},q^{2(f_k+j-1)})
$
for $1\leq k\leq m$.
Then, by~\eqref{partialQ1},
\begin{align*}
Q_i(x)
=\prod_{1\leq k\leq m}\prod_{1\leq j\leq\la^{(k)}_1, \la^{(k)'}_j\geq i}(1-q^{2(f_k+j-i)}x)
=\prod_{1\leq k\leq m}\prod_{1\leq j\leq\la^{(k)}_i}(1-q^{2(f_k+j-i)}x).
\end{align*}
Now, (2) follows from Theorem~\ref{thm:identify}
and Proposition~\ref{prop:DF&DD}.
\end{proof}

%
Note that,
in the case $m=1$ and $u_1=1$, Theorem~\ref{thm:K&D} recovers \cite[Theorem 7.2]{DF}.

\begin{rem}\label{cr-reading} By Theorems \ref{thm:VsDla} and \ref{thm:K&D},  the functor $\mc F$ in \eqref{sF} has the following combinatorial description.
For $\mathfrak s\in\mathscr S_r^{(n),\Z}$, there is
 a standard multipartition $\udla'\in\mc K^s_{f^*,(n)}(r)$ as constructed in \eqref{ga_s} such that the inflation of $D^{\udla}$ is the irreducible $\sH_\vtg(r)_{\mathbb C}$-module $V_{\mathfrak s}$ and $\mathfrak s=\overline{\bfs^\sfc_{\udla;\udu_f}}$  is the multisegment associated with the column residual segments \eqref{scla} of $\udla$. The image $\mc F(V_{\mathfrak s})$ is the irreducible $\rmU_q(\afgl)$-module $\Og_{(n)}^{\ot r}\ot_{\afHr} D^{\udla}$ that is isomorphic to the irreducible polynomial representation  $L(\bfQ_{\mathfrak s})$, where the dominant Drinfeld polynomials $\bfQ_{\mathfrak s}$ are defined by the row residual segments \eqref{srla} of $\udla$. This column-reading vs. row-reading of residual segments characterises the combinatorial feature of  the affine Schur--Weyl duality.
\end{rem}

\noindent
{\bf Acknowledgement.} The second author would like to thank Weiqiang Wang, Alexander Kleshchev, and Jun Hu for some helpful discussions.

\end{document}